\documentclass[11pt]{article}
\usepackage{amssymb}
\usepackage{graphicx}
\usepackage{xcolor} 
\usepackage{tensor}
\usepackage{fullpage} 
\usepackage{amsmath}
\usepackage{amsthm}
\usepackage{verbatim}
\usepackage{hyperref}
\usepackage{ulem}
\usepackage{subcaption}

\usepackage{enumitem}
\setlist[enumerate]{leftmargin=1.2em}
\setlist[itemize]{leftmargin=1.2em}

\definecolor{green}{rgb}{0,0.8,0} 
\usepackage{hyperref}

\newtheorem{theorem}{Theorem}[section]
\newtheorem{corollary}[theorem]{Corollary}
\newtheorem{lemma}[theorem]{Lemma}
\newtheorem{proposition}[theorem]{Proposition}

\theoremstyle{definition}

\theoremstyle{remark}
\newtheorem{remark}[theorem]{Remark}

\numberwithin{equation}{section}
\newcommand{\nrm}[1]{\Vert#1\Vert}

\newcommand{\nnrm}[1]{{\vert\kern-0.25ex\vert\kern-0.25ex\vert #1 
		\vert\kern-0.25ex\vert\kern-0.25ex\vert}}

\newcommand{\supp}{{\mathrm{supp}}\,}

\newcommand{\lap}{\Delta}

\newcommand{\ud}{\mathrm{d}}
\newcommand{\rd}{\partial}
\newcommand{\nb}{\nabla}

\newcommand{\alp}{\alpha}
\newcommand{\bt}{\beta}
\newcommand{\gmm}{\gamma}
\newcommand{\Gmm}{\Gamma}

\newcommand{\eps}{\epsilon}

\newcommand{\lmb}{\lambda}

\newcommand{\tht}{\theta}

\newcommand{\omg}{\omega}
\newcommand{\Omg}{\Omega}

\newcommand{\zt}{\zeta}

\newcommand{\bbR}{\mathbb R}

\newcommand{\bbT}{\mathbb T}

\newcommand{\calE}{\mathcal E}
\newcommand{\calF}{\mathcal F}

\newcommand{\calH}{\mathcal H}

\def\ct{\cos(\theta)}

\def\bX{\bar{X}}
\def\bY{\bar{Y}}
\def\bR{\bar{r}}

\def\bz{\bar{z}}

\begin{document}
	
\title{On vortex stretching for anti-parallel axisymmetric flows} 
\author{Kyudong Choi\thanks{Department of Mathematical Sciences, Ulsan National Institute of Science and Technology, kchoi@unist.ac.kr} 
	\and In-Jee Jeong\thanks{Department of Mathematical Sciences and RIM, Seoul National University, injee\_j@snu.ac.kr}
}
\date\today
 
\maketitle 
 
\begin{abstract} 
	We consider axisymmetric incompressible inviscid flows without swirl in $\mathbb{R}^3$, under the assumption that the axial vorticity is non-positive in the upper half space and odd in the last coordinate, which corresponds to the flow setup for head-on collision of anti-parallel vortex rings. For any such data, we establish monotonicity and infinite growth of the vorticity impulse on the upper half-space. As an application, we achieve infinite growth of Sobolev norms for certain   classical/smooth and compactly supported vorticity solutions in $\bbR^{3}$. 
\end{abstract}

\tableofcontents

\section{Introduction}

We are concerned with the Cauchy problem of the three-dimensional Euler equations in $\bbR^3$, which governs the dynamics of incompressible and inviscid fluids. In terms of the vorticity, the equations read
\begin{equation}\label{eq:Euler}
	\left\{
	\begin{aligned}
		& \rd_t \omg + u\cdot\nb \omg = \omg\cdot\nb u,  \\
		& u = \nabla\times (-\lap)^{-1}\omg ,
	\end{aligned}
	\right.
\end{equation} with $u(t,\cdot), \omg(t,\cdot): \bbR^3\to\bbR^3$ denoting the velocity and vorticity of the fluid, respectively. The right-hand side of the first equation in \eqref{eq:Euler} is commonly referred to as the vortex stretching term, as it can increase the size of $|\omg(t,\cdot)|$ along the flow trajectories. For Euler flows in $\bbR^3$, it is a very interesting problem to understand what types of flow configurations can lead to large growth of the vorticity maximum, especially in view of the Beale--Kato--Majda criterion which states that smooth solutions can blow up in finite time if and only if the quantity $\nrm{\omg(t,\cdot)}_{L^\infty}$ does. In the current work, we consider the specific scenario of head-on collision of anti-parallel vortex rings, which has been extensively studied using experimental and numerical methods. For such flow configurations, we obtain maximum principles and infinite growth for physically natural averaged quantities involving the vorticity. As a byproduct, we are able to deduce infinite growth of the vorticity $L^\infty$ norm for certain globally defined \textit{classical} solutions to \eqref{eq:Euler} (see the estimate \eqref{ex_linfty_gro}), among others. For some $C^\infty$--\textit{smooth} solutions, we obtain infinite growth of   the vorticity  $C^\alpha$ H\"older norm (see Theorem \ref{thm:gradient-growth}).

 
 \subsection{Main results}
 
From now on, we shall consider the specific class of Euler flows satisfying the assumption of \textit{axisymmetric without swirl} in $\mathbb{R}^3$. In this case, the Euler equations \eqref{eq:Euler} reduce to  
 \begin{equation}\label{eq:Euler-axisym-no-swirl}
	\begin{split}
		\rd_t \omg + u \cdot \nb \omg = \frac{u^r}{r} \omg  
	\end{split}
\end{equation} where $\omg(t,\bold{x})=\omg^\tht(t,r,z)e_\tht(\tht)$, $u \cdot \nb := u^r(t,r,z) \rd_r + u^z(t,r,z) \rd_z$, and $(r,\tht,z)$ is the cylindrical coordinate system in $\bbR^3$. For the axisymmetric Euler equations without swirl \eqref{eq:Euler-axisym-no-swirl}, 
  there is a unique global-in-time solution $\omg(t,\cdot)$   satisfying $\omg, r^{-1}\omg \in L^\infty([0,T]; L^1 \cap L^\infty(\mathbb{R}^3) )$ for any $T>0$  {(by
\cite{UI}, also see \cite{Raymond, Danchin}),} whenever 
the initial vorticity
\begin{equation}\label{init_assum}
  \omg_0=\omg^\tht_0(r,z)e_\tht(\tht)  \quad\mbox{satisfies}\quad  \omg_0, \,\frac{\omg_0}{r} \in L^1 \cap L^\infty(\mathbb{R}^3).
\end{equation} However, note that unlike the case of the Euler equations in a two-dimensional domain, the vortex stretching term is still present in \eqref{eq:Euler-axisym-no-swirl}, and therefore it is possible that the $L^p$--norms of the vorticity could blow up in infinite time. This type of phenomenon has been actually observed in the collision of a pair of counter-rotating vortex rings. To model this flow configuration, we shall work with vorticities which are odd in $z$ (``anti-parallel'') and non-positive on {the upper half-space $\bbR^3_+:=\{z>0\}$}; that is, $\omg^\tht$ satisfying \begin{equation}\label{eq:vort-assumption}
	\begin{split}
		\omg^\tht(t,r,z)=-\omg^\tht(t,r,-z),\qquad \omg^\tht(t,r,z) \le 0 \quad\mbox{for}\quad z\ge0.
	\end{split}
\end{equation} Then, it is not difficult to see that the assumptions in \eqref{eq:vort-assumption} are satisfied for the unique solution, if they hold for the initial data $\omg^\tht_0$. Furthermore, \eqref{eq:Euler-axisym-no-swirl} reduces to a system posed on the upper half-space $\mathbb{R}^3_+$, and we shall often regard {$\omg^\tht$ (and hence $\omg$)} as defined on $\mathbb{R}^3_+$. 
Our first main result establishes a universal lower bound for the growth rate of the vorticity impulse and support diameter for \textit{any such} compactly support initial data.

\begin{theorem}[Impulse and diameter growth]\label{thm:main}
	Together with  {\eqref{init_assum} and} \eqref{eq:vort-assumption}, assume that the initial vorticity satisfies \begin{equation*}
		\begin{split}
			0<\iint_{[0,\infty)^2}  -z  {\omg^\tht_0} (r,z) \, \ud r \ud z < + \infty  \quad \mbox{and} \quad  0<  \iint_{[0,\infty)^2} -r^2\omg^\tht_0(r,z) \, \ud r \ud z < + \infty.
		\end{split}
	\end{equation*}
	Then, the unique global solution $\omg(t,\cdot)$ to \eqref{eq:Euler-axisym-no-swirl} satisfies that 
  {\begin{equation*}
		\begin{split}
			\iint_{[0,\infty)^2} -z\omg^\tht(t,r,z) \, \ud r \ud z \quad\mbox{is strictly decreasing in time,}
		\end{split}
	\end{equation*}	}
	\begin{equation*}
		\begin{split}
			\iint_{[0,\infty)^2} -r^2\omg^\tht(t,r,z) \, \ud r \ud z \quad\mbox{is strictly increasing in time,}
		\end{split}
	\end{equation*} and for any $\varepsilon>0$,  \begin{equation*}\label{eq:main-growth}
		\begin{split}
			 \iint_{[0,\infty)^2} -r^2\omg^\tht(t,r,z) \, \ud r \ud z \ge C_\varepsilon (1 + t)^{\frac{2}{15}-\varepsilon}  \quad \mbox{for all} \quad t\ge0,
		\end{split}
	\end{equation*}  where $C_\varepsilon>0$ is a constant depending on $\varepsilon$ and $\omg_0$.  {In particular, from \eqref{eq:main-growth} we deduce that \begin{equation*}
	\begin{split}
		\sup \{ \sqrt{x_1^2+x_2^2} \,:\, \bold{x}=(x_1,x_2,x_3) \in \supp\, \omg(t,\cdot) \} \ge  C_\varepsilon (1 + t)^{\frac{1}{15}-\varepsilon}  \quad \mbox{for all} \quad t\ge0
	\end{split}
\end{equation*} holds whenever $\omg_0(\bold{x})$ is compactly supported in $r=\sqrt{x_1^2+x_2^2}$. 
}
\end{theorem} 

\medskip

\noindent In the proof, we shall specify the dependence of $C_\varepsilon$ in $\omg_0$, which is somewhat complicated but only involves controlled quantities by Euler dynamics. We shall now illustrate several consequences of the above result. As a direct application, we are able to deduce infinite growth of the $L^p$--norms of the vorticity, 
under an additional assumption on the initial vorticity. 

\medskip

\noindent 

\begin{corollary}\label{cor:main2}
	Let $0\leq \delta<1/15$. Assume that for some $p \in [2-\delta,\infty]$, the initial data satisfies
	\begin{equation}\label{eq:vort-assumption-growth-p2}
		\begin{split}
			\left\Vert\frac{r}{|\omg_0|}\mathbf{1}_{\{ |\omg_0|>0 \}}\right\Vert_{L^{\frac{1-\delta}{1-((2-\delta)/p)}} (\bbR^3) }
			< \infty \quad 
			\mbox{and that\quad  $\omg_0(\bold{x})$ is compactly supported in $r $ }\quad
		\end{split}
	\end{equation} in addition to the hypotheses of Theorem \ref{thm:main}. Then, for each $\varepsilon>0$, we have \begin{equation*}
	\begin{split}
		\nrm{\omg(t,\cdot)}_{L^p(\mathbb{R}^3)} \ge C_\varepsilon (1+t)^{\frac{\frac{2}{15}-2\delta}{2-\delta}-\varepsilon}   \quad \mbox{for all} \quad t\ge0,
	\end{split}
\end{equation*}  where $C_\varepsilon>0$ is a constant depending on $\varepsilon$ and $\omg_0$. 
\end{corollary}

\medskip

\begin{remark}\label{rem_patch}\label{rem_strange} We remark on the classes of vorticities satisfying the assumption \eqref{eq:vort-assumption-growth-p2}, which may look a bit strange. 
	\begin{itemize}
		\item \textbf{Vortex patches}: By a vortex patch solution, we mean {$\omg=\omg^\tht e_\tht$}  of the form \begin{equation}\label{eq:patch2}
			\begin{split}
				\omg^\tht(t,r,z) = \sum_{i=1}^{n} -a_i(t,r,z) \mathbf{1}_{\Omg_i(t)}(r,z) \quad \mbox{on} \quad (r,z) \in [0,\infty)^2\, . 
			\end{split}
		\end{equation} Here, $\Omg_i$ are bounded axisymmetric open sets in $\mathbb{R}^3$ which are disjoint with each other, and $a_i$ are some bounded non-negative functions. Then, the assumptions of Corollary \ref{cor:main2} are satisfied for all $p\in[2,\infty]$ (with $\delta=0$) for \eqref{eq:patch2} as long as the sets $\Omg_i$ are separated from $\{ r = 0 \}$ and the functions $a_i$ are bounded away from 0, at the initial time. The boundary of $\Omg_i$ and $a_i$ can be taken to be $C^\infty$--smooth; it is well-known that for axisymmetric initial data of the form \eqref{eq:patch2}, $C^\infty$--regularity of $a_i$ and $\partial\Omg_i$ propagates globally in time (\cite{Hu,GaSa,S-patch}), with associated velocity field $C^\infty$--smooth in $\overline{\Omg_i}$ for each $i$. Therefore, in this class of initial data, obtain 
		\begin{equation}\label{patch_gro}
				 \max_i \nrm{a_i(t,\cdot)}_{L^p(\mathbb{R}^3)} \gtrsim t^{\frac 1 {15}-} \quad \mbox{for any}\quad p\ge2\quad\mbox{and}\quad \max_i \mathrm{diam}_{\mathbb{R}^3}\, \Omg_i(t,\cdot) \gtrsim t^{\frac 1 {15}-}. 
		\end{equation}
		\item \textbf{Vorticities with $C^{1,\gmm}$--regularity}: To apply the above result to smoother vorticities, we first observe that 
	for	compactly supported and continuous $f$, its reciprocal  can belong to $L^q,\, q>0$ on the set $\{|f|>0\}$  \textit{i.e.} $$\int |f|^{-q}\mathbf{1}_{\{|f|>0\}}\ud x<\infty$$  when the function $f$ touches the boundary of its support by the rate $x^\alpha$ with $$\alpha\cdot q <1.$$  
		For instance, in the case $p = \infty$, for each 
		$0<\gamma<{{1}/{14}}$,
		the assumption \eqref{eq:vort-assumption-growth-p2} holds for some compactly-supported (away from the symmetry axis $\{r=0\}$) $C^{1,\gamma}$ vorticities $\omg_0$ (with any choice  of   $\delta\in(\gamma/(1+\gamma),1/15)$). As a consequence, the solutions from such data $\omega_0\in C_c^{1,\gamma}(\mathbb{R}^3\setminus\{r=0\})$ 
		satisfy, for some $\alpha=\alpha(\gamma)>0$,
		\begin{equation}\label{ex_linfty_gro}
			\begin{split}
				\nrm{\omg(t,\cdot)}_{L^\infty(\mathbb{R}^3)} \ge C  (1+t)^{\alpha},   \quad t\ge0.
			\end{split}
		\end{equation} In a similar vein, in the case of $p=2$, there are some   vorticities $\omg_0\in C_c^{0,\gamma}(\mathbb{R}^3\setminus\{r=0\})$ with any $0<\gamma<1/28$ satisfying \eqref{eq:vort-assumption-growth-p2} so that
	\begin{equation}\label{ex_l2_growh}
	\nrm{\omg(t,\cdot)}_{L^2(\mathbb{R}^3)} \ge C  (1+t)^{\alpha},   \quad t\ge0
\end{equation} for some $\alpha>0$ depending on $\gamma$. In Section \ref{sec_navier}, we connect this infinite enstrophy  growth for Euler into   {enstrophy inflation} for Navier--Stokes with small viscosity. 
	\end{itemize} 
\end{remark}
\medskip 

To the best of our knowledge, this is the first construction of initial    vorticity $\omg_0$, which is compactly supported and is more regular than $C^1$ (so that 
the vorticity solution solves \eqref{eq:Euler} in the \textit{classical} sense and
the velocity $u$ is more regular than $C^2$), with infinite growth of the vorticity \textit{maximum} as $t\to\infty$, for the incompressible Euler equations in $\mathbb{R}^3$. Furthermore, we are not aware of any previous results which gives infinite growth \eqref{patch_gro} of the support diameter and $L^\infty$--norm of $\omg$ for \textit{smooth} vortex patches.

\medskip

\noindent 
Let us now present some further results which can be obtained based on Theorem \ref{thm:main}. To begin with, we show that both the  growth rate of $\beta_1(\delta):={\frac{\frac{2}{15}-2\delta}{2-\delta}}$ and the range of $\delta\in [0, 1/15)$ from Corollary \ref{cor:main2} can be upgraded, if one takes $\limsup$ in time. In the case of $L^\infty$, we have the following result: 
\begin{theorem}[Enhanced growth rate]\label{thm:main3}
	Assuming that \eqref{eq:vort-assumption-growth-p2} holds for $p = +\infty$ and for some $\delta\in[0,3/20)$  in addition to the hypotheses of Theorem \ref{thm:main}, we have 
some $\bt_0 = \bt_0(\delta)>\max\{\beta_1(\delta),0\}$ satisfying 
	 \begin{equation*}\label{eq:main-growth-inf3}
		\begin{split}
			\limsup_{t\to\infty} \frac{\nrm{\omg(t,\cdot)}_{L^{\infty}(\mathbb{R}^3)} }{(1 + t)^{\bt_0-\varepsilon} } =+\infty, \qquad  \mbox{for any } \varepsilon>0. 
		\end{split}
	\end{equation*}  
\end{theorem}    A version of Theorem \ref{thm:main3}  for \textit{finite} $p$ can also be derived  with some different $\bt_0$ and $\delta$ depending on $p$. Similarly as in Remark \ref{rem_strange}, the above theorem guarantees the existence of data
$  \omega_0\in C_c^{1,\gamma}(\mathbb{R}^3\setminus\{r=0\})$ for each
 $0<\gamma<3/17$ whose   solution to \eqref{eq:Euler} satisfies $\limsup_{t\to\infty}\nrm{\omg(t,\cdot)}_{L^\infty}=\infty.$

\medskip

\noindent Lastly, recall that we proved infinite $L^\infty$--norm growth of the vorticity either when the initial data were a patch or $C^{1,\gmm}$--smooth. On the other hand, we are able to construct a class of $C^\infty$--initial data for which infinite growth of the $C^\alp$--norm occurs for any $\alp>0$ as $t\to\infty$.
\begin{theorem}[Infinite gradient growth]\label{thm:gradient-growth}
	There exists an initial datum $\omg_{0} \in C^\infty_c(\bbR^3)$ whose associated Euler solution to \eqref{eq:Euler} satisfies \begin{equation*}
		\begin{split}
			\limsup_{t\to\infty}\nrm{\omg(t,\cdot)}_{C^\alp(\bbR^3)} = +\infty, \qquad  \mbox{for any } \alp>0. 
		\end{split}
	\end{equation*} 
\end{theorem}

\subsection{Discussion}

Let us present some motivation for studying vorticities of the form \eqref{eq:vort-assumption} as well as related previous works. 

\medskip

\noindent \textbf{Head-on collision of antiparallel vortex rings}. Our primary motivation for studying axisymmetric vorticity satisfying \eqref{eq:vort-assumption} comes from the study of head-on collision of vortex rings (\cite{KaMi,KaMy,ShLe,CWCCC,SLZF,LN,SSH,PeRi,Os,CLL,GWRW,SMO}). Since a single vortex ring with negative axial vorticity travels downward, two vortex rings with opposite sign will ``collide'' with each other if the ring with negative vorticity is on the above of the other. 
We illustrate the setup in Figure 
 \ref{fig_lab} where the dashed lines represent direction of the flow. \begin{figure}
\centering
\begin{subfigure}{.5\textwidth}
  \centering
  \includegraphics[width=1\linewidth]{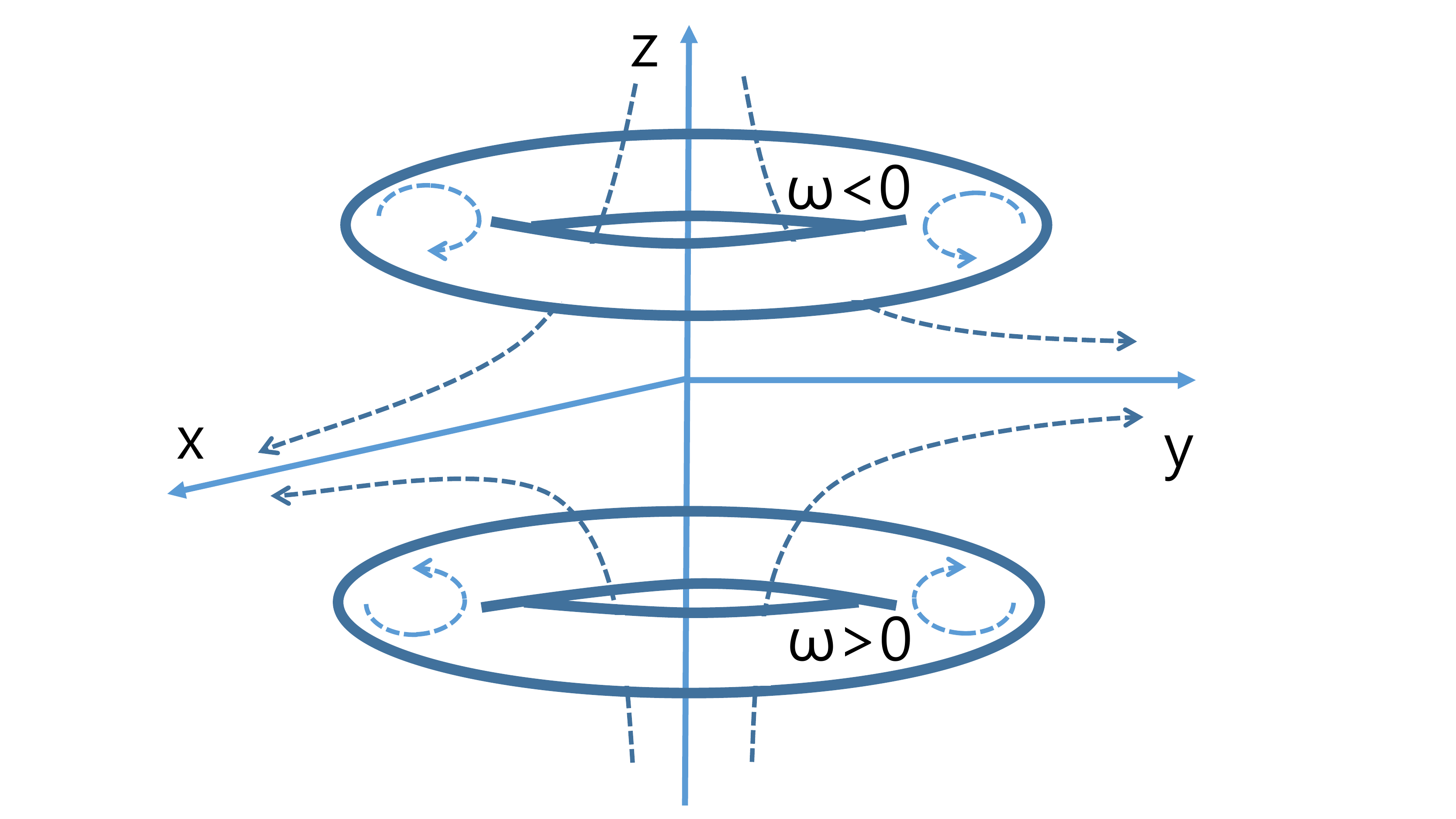}
  \caption{in $xyz$-space}
  \label{fig:sub1}
\end{subfigure}%
\begin{subfigure}{.5\textwidth}
  \centering
  \includegraphics[width=1\linewidth]{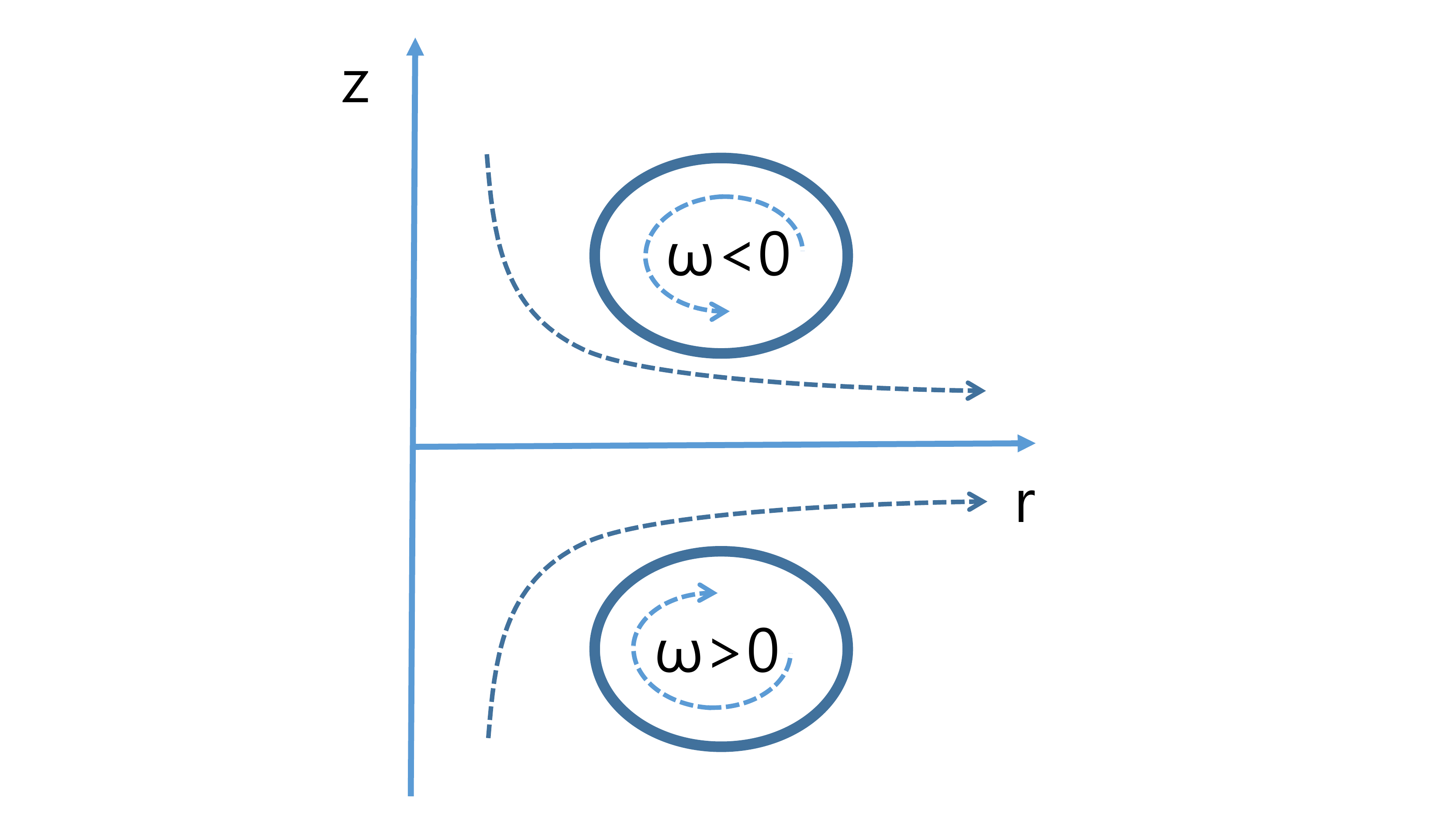}
  \caption{in the cross-section in $rz$-plane}
  \label{fig:sub2}
\end{subfigure}
\caption{An illustration on  antiparallel vortex rings}
\label{fig_lab}
\end{figure}
The interested readers may find the flow visualizations in \cite{youtube1} as well as in aforementioned references. 
This flow setup exhibits several fascinating features including strong vortex stretching, bulging instability, rebound effects, breakdown of axisymmetry and vortex reconnection, and has been studied extensively using experimental, theoretical, and computational methods. The authors in \cite{CWCCC} demonstrate that the evolution of the vortex rings goes through three stages, before their breakdown: (i) free-traveling, (ii) vortex stretching, and (iii) viscous dissipation. This is intuitively clear: initially, the rings are separated and they travel towards each other as the interaction is weak. As they become sufficiently close, the rings start to stretch out to radial direction due to incompressibility. Numerical simulations and experiments clearly show that as the Reynolds number becomes large, the rings get stretched further away from the symmetry axis before viscous dissipation becomes dominant. In the current work, we employ the inviscid equation and focus on the vortex stretching phenomenon for large times. The basic question we try to answer is: what is the rate of vorticity growth as the time goes to infinity? To answer this question, one can hope to gain some insights from the simpler case of two-dimensional flows. 

\medskip

\noindent \textbf{Vortex dynamics in the two-dimensional case}. Under the odd symmetry assumption in $z$, axisymmetric Euler flows without swirl share structural similarities with two-dimensional Euler flows subject to the odd-odd symmetry: the planar vorticity $\omg$ satisfying $\omg(x_1,x_2) = -\omg(-x_1,x_2)=-\omg(x_1,-x_2)$ in $\bbR^2$. Assuming that the vorticity is non-negative on the first quadrant $(\bbR_+)^2$, it can be shown that the overall flow is directed southeast on $(\bbR_+)^2$. If one takes a single point vortex with positive sign on $(\bbR_+)^2$ and extend it to $\bbR^2$ with odd-odd symmetry, then one can directly compute that the vortex on the first quadrant goes out to the  $x_1$-axis linearly in time while its second component asymptotes to a positive constant (\cite{Yang}). Even for general non-negative vorticities on $(\bbR_+)^2$, a similar result is available (\cite{ISG99}): the center of vorticity $\int_{(\bbR_+)^2} x_1 \omg(t,x) dx $ grows linearly for all positive times. We note that this odd-odd scenario has been used to prove growth of $|\nb\omg|$ in two-dimensional Euler flows (\cite{Den2,Den3,KS,Z,KiLi,J-JMFM}).

\medskip

\noindent \textbf{Dyson model for vortex rings}. Returning to axisymmetric flows and the problem of finding the rate of vortex stretching, it is natural to first consider the dynamics of circular vortex filaments (thin-cored vortex rings), which are the axisymmetric versions of point vortices in 2D. The well-known difficulty in the filament case is that for a fixed circulation $\Gmm$ and distance to the axis $R$, the self-induction velocity diverges as the core radius $a$ goes to zero: the calculation attributed to Kelvin (\cite{Thomson,HH,Lamb}) shows that for $a \ll R$, \begin{equation*}
	\begin{split}
		u^z = \frac{\Gmm}{4\pi R} \left( \ln\frac{8R}{a} - \frac{1}{4} + O(\frac{a}{R}) \right). 
	\end{split}
\end{equation*} To handle this difficulty, the usual assumption is that the core of each vortex ring remains circular with a \textit{fixed} radius; this is the \textit{localized induction approximation}. Then, proceeding similarly as in the 2D case, 
one can derive a system of ordinary differential equations for the motion of thin cored vortex rings (see \cite{Lamb,Yang} for instance); it is commonly referred to as Dyson's model in the literature (\cite{Dyson}). Applying Dyson's model to the system of two antiparallel vortex rings, one obtains that the dynamics is essentially the same as the point vortex motion under the odd-odd symmetry, which suggests $\nrm{\omg(t,\cdot)}_{L^\infty}\sim t$. However, as the rings move away from the symmetry axis, vortex stretching together with conservation of total vorticity $\nrm{r^{-1}\omg(t,\cdot)}_{L^1(\mathbb{R}^3)}$ forces that the core area in the $(r,z)$--plane must vanish to zero: this shows that the localized induction approximation is \textbf{self-inconsistent}. This suggests that the asymptotic behavior as $t\to\infty$ for antiparallel vortex rings could be different from the two-dimensional case and more difficult to understand. 

\medskip 

\noindent \textbf{Childress model of vortex growth}. In a series of works (\cite{Child07,Child08,ChildGil}), Childress and his collaborators investigated the exact same problem: what is the rate of vortex stretching in axisymmetric flows under the assumption \eqref{eq:vort-assumption}? Due to the conservation of $r^{-1}\omg$ along particle trajectories, for compactly supported vorticities, the question is roughly the same with finding the growth rate of the vorticity support in $r$. The naive a priori estimates give that the rate is bounded by $e^{Ct}$, with $C>0$ depending on the initial data. In \cite{Child07,Child08}, Childress shows that this upper bound can be improved to $t^2$ and $t^{4/3}$ by solving certain maximization problems using the constraint of support volume and kinetic energy, respectively. In these works Childress imposes the conditions \eqref{eq:vort-assumption} on vorticity, among others.\footnote{We note that the upper bound of $t^2$ can be proved in general using a rather recent estimate of Feng--\v{S}ver\'{a}k  \cite{FS}, see Lemma \ref{lem:FS} below.} The $t^{4/3}$ bound can be seen heuristically as follows (see \cite{Child08,ChildGil}): assuming $\omg(t,r,z) \simeq R(t)\Omg(a(t)^{-1}(r,z) -( R(t),0))$ with some profile $\Omg$ and $R(t) > a(t) > 0$, the kinetic energy  conservation dictates the scaling $R^3 a^4 \sim 1  $ while the Euler evolution forces $\dot{R} \sim Ra$. While this ansatz clearly contradicts the conservation of circulation, Childress argues that the (possible) formation of a long tail accounts for the loss of circulation. Numerical simulations from \cite{ChildGil} suggest that growth rate of $t^{4/3}$ can be indeed achieved, with asymptotic profile $\Omg$ given by the Sadovskii vortex. 

\medskip

\noindent \textbf{Growth of vorticity for three-dimensional Euler}. There has been significant interest in the construction of Euler flows \eqref{eq:Euler} with vorticity maximum growing in time. Unlike the 2D case, $\nrm{\omg(t,\cdot)}_{L^\infty}$ could increase when the vorticity aligns with an eigenvector of $\nb u$ having a positive eigenvalue. For smooth and decaying solutions to the 3D Euler equations, it is well-known that finite time singularity formation could occur if and only if $\nrm{\omg(t,\cdot)}_{L^\infty}$ blows up in finite time. A few finite-time blow-up results for finite-energy solutions to 3D Euler exist (\cite{EJE,EJO,E3,EGM,ChenHou})
but in these results either the presence of physical boundaries or lack of smoothness\footnote{{For instance, \cite{E3, EGM} for the case of $\mathbb{R}^3$ do not meet  the initial condition \eqref{init_assum}.}} of vorticity plays an important role in the growth. Even when the vorticity does not blow up, it is an interesting question to understand the possible rate of vortex stretching in various situations. If one considers axisymmetric domains with a boundary, then one can obtain growth of the vorticity maximum along the boundary \cite{Y3}. Even without physical boundaries, rather simple examples of vorticity growth can be obtained using the so-called $2+\frac{1}{2}$ dimensional flow construction, but this requires the physical domain to be in the form $\Omg\times\bbT$ where $\Omg$ is a two-dimensional domain (\cite{Y3}). Even in the case of $\bbT^3$, we are not aware of any results giving infinite vorticity growth for smooth vorticity, not relying on the $2+\frac{1}{2}$ dimensional geometry (however see \cite{Pe5}). In our previous work \cite{CJ_Hill}, we obtained arbitrarily large but finite growth of $\nrm{\omg(t,\cdot)}_{L^\infty}$ in $\mathbb{R}^3$ using perturbations of the Hill's vortex  (also see \cite{Choi2020} for stability of the vortex). In the same work, infinite growth of $\nrm{\nb^2\omg(t,\cdot)}_{L^\infty}$ was obtained for smooth and compactly supported initial data. Very recent numerical computations by Hou suggests finite time singularity for axisymmetric Euler with swirl in the interior of the domain \cite{Hou}. 

\medskip

\noindent \textbf{Large enstrophy growth for the Navier--Stokes equations}.  There has been some interest in the question of possible growth of the enstrophy, $L^2$-norm of the vorticity, in the three-dimensional Navier--Stokes equations as well as related dissipative systems (\cite{FJL,LuLu,LD,Pel,PoPr,DoGi,Pro1,Pro2,Pro3}). This quantity is particularly interesting as it determines regularity and uniqueness of the Navier--Stokes solutions. So far, only \textbf{upper bounds} on the enstrophy are available: in Lu--Doering \cite{LD}, the authors have obtained the bound \begin{equation*}
	\begin{split}
		\frac{d}{dt} \nrm{\omg(t,\cdot)}_{L^2} \le \frac{C}{\nu^3}\nrm{\omg(t,\cdot)}_{L^2}^3.
	\end{split}
\end{equation*} Interestingly, in \cite{LD}, it is reported that the maximal growth of the enstrophy seems to be achieved by flows in which a pair of vortex rings are colliding with each other, which are exactly the flows considered in the current work. Later in Section \ref{sec_navier}, we show that the infinite growth of the enstrophy for Euler flows can be translated to large enstrophy growth for Navier--Stokes flows. 

\subsection{Outline of the proof}
Let us outline the proof of Theorem \ref{thm:main}, which is inspired by a related work of Iftimie--Sideris--Gamblin \cite{ISG99}. Among others, the authors prove that for compactly supported, non-negative, and non-trivial initial vorticity on the positive quadrant $(\bbR_+)^2$ (extended to $\bbR^2$ by odd-odd symmetry as described above), the 2D Euler solution satisfies \begin{equation}\label{eq:ISG}
	\begin{split}
		\iint_{(\bbR_+)^2} x_1\omg(t,\bold{x}) \, \ud \bold{x}  \ge ct,
	\end{split}
\end{equation} with $c = c(\omg_0)>0$. Their proof consists of two main steps: the first is to establish the monotonicity property \begin{equation*}
\begin{split}
	\frac{d}{dt}\iint_{(\bbR_+)^2} x_1\omg(t,\bold{x}) \, \ud \bold{x} \ge 0 , \qquad \frac{d}{dt}\iint_{(\bbR_+)^2} x_2\omg(t,\bold{x}) \, \ud \bold{x} \le 0. 
\end{split}
\end{equation*} Then, the second step is to prove a lower bound on the kinetic energy $E = \iint_{(\bbR_+)^2} |u|^2 \, \ud \bold{x} $ of the form \begin{equation*}
\begin{split}
	E \le C \left( \iint_{(\bbR_+)^2} x_2\omg(t,\bold{x}) \, \ud \bold{x}  \right)^p \left( \frac{d}{dt}\iint_{(\bbR_+)^2} x_1\omg(t,\bold{x}) \, \ud \bold{x} \right)^{q}
\end{split}
\end{equation*} with some $p,q>0$ and $C>0$ depending on the initial data. Then, using conservation of $E$ and monotonicity of $\iint_{(\bbR_+)^2} x_2\omg(t,\bold{x}) \, \ud \bold{x} $, we conclude that \begin{equation*}
\begin{split}
	\left( \iint_{(\bbR_+)^2} x_2\omg_0(\bold{x}) \, \ud \bold{x}  \right)^{-p}\frac{E}{C} \le \left( \frac{d}{dt}\iint_{(\bbR_+)^2} x_1\omg(t,\bold{x}) \, \ud \bold{x} \right)^{q},
\end{split}
\end{equation*} from which \eqref{eq:ISG} follows. 

\medskip

\noindent While we follow the above steps, there are some difficulties arising in our case due to the form of the axisymmetric Biot--Savart law and the presence of vortex stretching.
\begin{itemize}
	\item \textbf{Monotonicity lemmas}: It turns out that in the axisymmetric case, the monotone quantities are \begin{equation*}
		\begin{split}
			P(t) = \iint_{(\bbR_+)^2}  -r^2 \omg^\tht(t,r,z) \, \ud r \ud z, \qquad  Z(t) = \iint_{(\bbR_+)^2} -z \omg^\tht(t,r,z) \, \ud r \ud z ; 
		\end{split}
	\end{equation*}  the former increases monotonically in time while the later decreases. While this tendency of vortex dynamics is clear from numerous simulations and computations, we were not able to find a rigorous proof in the literature. In the meanwhile, this confirms that the ``rebound effect'' of vortex rings is a purely viscous phenomenon as argued in \cite{CWCCC}; the convection term only brings the antiparallel rings towards each other. Unlike the two-dimensional case, the proof of \begin{equation*}
	\begin{split}
		\frac{d}{dt} \iint_{(\bbR_+)^2} -z \omg^\tht(t,r,z) \, \ud r \ud z  \le 0
	\end{split}
\end{equation*} is highly nontrivial and we could establish it only by using the formulation of the Biot--Savart law in terms of certain elliptic integrals  {(see Lemma \ref{lem:z-down}). }
	
	\item \textbf{Kinetic energy bound}: Following the ideas of Iftimie--Sideris--Gamblin, we try to bound $E$ using $\dot{P} = \frac{d}{dt}P$ and other conserved quantities. However, inspecting the expressions for $E$ and $\dot{P}$ given in Lemma \ref{lem:E-P-bounds}, it seems impossible to obtain a bound of the form $E \lesssim (\dot{P})^{a}$ for some $a>0$, unlike the 2D case. Indeed, $E$ grows faster than $\dot{P}$ if one places the initial vorticity further away from the axis, keeping the $L^p$ norms of $r^{-1}\omg_0$ under control. To handle this issue, we use $P$ itself, which controls the mass of vorticity far from the axis: our key estimate is \begin{equation*}
		\begin{split}
			E \lesssim \dot{P}^{1-\frac{1}{q}} P^{\frac{1}{q}}, 
		\end{split}
	\end{equation*} for any $q <\frac{15}{13}$. Integrating this in time gives the bound $P(t) \gtrsim 1 + t^{\frac{2}{15}-}$. 
	
	\item \textbf{Relation between $P$ and vorticity maximum}: Growth of $P$ implies that, in some averaged sense, the center of vorticity in the radial direction increases. Recalling the Cauchy formula for axisymmetric flows without swirl, this implies growth of the vorticity maximum. Unfortunately, to deduce $\nrm{\omg(t,\cdot)}_{L^\infty} \to \infty $ from $P(t)\to\infty$, we need an additional assumption on the distribution function of the initial vorticity, which is expressed in \eqref{eq:vort-assumption-growth-p2}: heuristically, the measure of the set on which the vorticity takes very small values cannot be too large.   {Under this assumption
	\eqref{eq:vort-assumption-growth-p2} with $\delta=0$, which can be easily achieved by patch-type vorticities,} we obtain\begin{equation*}
		\begin{split}
			P^{\frac12} \lesssim \nrm{\omg}_{L^\infty}. 
		\end{split}
	\end{equation*} This implies $\nrm{\omg(t,\cdot)}_{L^\infty} \gtrsim t^{\frac{1}{15}-}$. 
 {	We note that the condition \eqref{eq:vort-assumption-growth-p2} with $\delta=0$ (and $p=\infty$)
can be satisfied for some $C^\gamma$--data	$\omega_0$ whenever $\gamma<1$ while it fails 
	  for typical $C^1$--data $\omega_0$. Fortunately, there is a little room for $\delta>0$ in the analysis of Lemma \ref{lem:P-ineq-Linfty2} providing the existence of  Lipschitz data $\omega_0$ with infinite growth.
	}
	\item \textbf{Upgrading growth rate   by contradiction}: We perform an additional estimate which enhances the growth rate of $\nrm{\omg(t,\cdot)}_{L^\infty}$. The idea is to \textit{assume} a hypothetical bound of the form $\nrm{\omg(t,\cdot)}_{L^\infty} \lesssim t^\bt$. For instance, when $\delta=0$, we get $\nrm{\omg(t,\cdot)}_{L^\infty} \lesssim P(t)^{\frac{15\bt}{2}+}$ using the lower bound $P(t) \gtrsim 1 + t^{\frac{2}{15}-}$. Then, we can use this new estimate to obtain an improved bound on the kinetic energy: \begin{equation*}
		\begin{split}
			E \lesssim \dot{P}^{1-\frac{1}{q}} P^{(\frac{1}{2} + \frac{45}{56}\bt  {-})\frac{1}{q}},
		\end{split}
	\end{equation*} for any $q< \frac{17}{14}$. This time, we obtain \begin{equation*}
	\begin{split}
		\nrm{\omg(t,\cdot)}_{L^\infty}  \gtrsim P^{\frac12} \gtrsim 1 + t^{\frac{1}{\frac{5}{3} + \frac{15}{8}\bt}-}. 
	\end{split}
\end{equation*} This is a contradiction to $\nrm{\omg(t,\cdot)}_{L^\infty} \lesssim t^\bt$ when $\bt<\bt_0=\frac{1}{45}(\sqrt{670}-20) \simeq 0.13$. 
 	\item \textbf{Infinite gradient growth}: The construction in Theorem \ref{thm:gradient-growth} of {compactly supported $C^\infty(\mathbb{R}^3)$--initial data 
 	$\omg_0(\bold{x})$} exhibiting growth of the $C^\alp$--norm of the vorticity requires combining several ideas from earlier works  with some new ingredients which are specific to the three-dimensional axisymmetric Euler equations. The basic idea is to place a smoothed out version of the Bahouri--Chemin solution (\cite{BC}) in the $(r,z)$--coordinate plane. This induces a strong hyperbolic flow near the origin, which stretches the vorticity gradient (see \cite{HeKi,EJ,BL1,Den,Den2,Z,KS,X,JY,KRYZ,JKim,Do,J-JMFM} for a few references utilizing this observation).
	 However, since our domain is \textit{unbounded}, most of the mass could escape to infinity which would slow down the rate of stretching at the origin.  
In the current setup, 	 we show that this is not allowed unless $\nrm{\omg(t,\cdot)}_{L^\infty}$ diverges to infinity--this is due to the presence of the vortex stretching term. Here, it is important that we can take the support of $\omg^\tht_0(r,z)$ (as a function on  the right half-plane $\{(r,z)\in\mathbb{R}^2\,:\, r\geq 0 \}$) to touch the axis $\{r=0\}$. Furthermore, under the contradiction hypothesis that the $C^\alp$--norm of the vorticity is uniformly bounded, we obtain a uniform-in-time control of the velocity gradient near the origin. This is then sufficient to achieve infinite vorticity gradient growth, proceeding as in Zlato{\v{s}} \cite{Z}.  
\end{itemize}


\subsection{Notation and conventions} From now on, for a simpler notation, we denote the scalar axial vorticity $\omega^\theta(r,z)$ simply by the symbol $\omega$ unless there might be confusion. We shall write $\Pi = \{ (r,z) : r \ge 0, z \in \bbR \}$ and $\Pi_+ =  \{ (r,z) : r, z \ge 0 \}$. Similarly, $\bbR^3_+ = \{ x \in \bbR^3 : x_3 \ge 0 \}$. As usual, $C$ will denote various positive absolute constants whose value could vary from a line to another. When $C$ depends on a few parameters, we will sometimes express the dependence using subscripts. Unless otherwise specified, the $L^p$ norms are always taken with respect to the Lebesgue measure on $\bbR^3$.


\section{Preliminaries}\label{sec:prelim} 

\subsection{Elliptic integrals}
It will be convenient to introduce some elliptic integral functions in order to simplify the axisymmetric Biot--Savart law. We begin with 
\begin{equation*}\label{eq:F}
\begin{split}
	\calF(s) & := \int_0^\pi \frac{\ct}{(2(1-\ct)+s)^\frac12} \, \ud \tht \\
	& = \frac{1}{(4+s)^{\frac12}}\left( (2+s)\calE_K\left( \frac{4}{4+s} \right) - (4+s)\calE_E\left( \frac{4}{4+s} \right) \right). 
\end{split}
\end{equation*} Then, \begin{equation}\label{eq:DF}
\begin{split}
	\calF'(s) & = -\int_0^\pi \frac{\ct}{2(2(1-\ct)+s)^\frac32} \, \ud \tht \\
	& = \frac{1}{2s(4+s)^{\frac12}}\left( s\calE_K\left( \frac{4}{4+s} \right) - (2+s)\calE_E\left( \frac{4}{4+s} \right) \right)
\end{split}
\end{equation} and \begin{equation}\label{eq:D2F}
\begin{split}
	\calF''(s)& = \int_0^\pi \frac{3\ct}{4(2(1-\ct)+s)^\frac52} \, \ud \tht \\
	&= \frac{1}{4s^2(4+s)^{\frac32}}\left( -s(2+s)\calE_K\left( \frac{4}{4+s} \right) + (16+4s+s^2)\calE_E\left( \frac{4}{4+s} \right) \right). 
\end{split}
\end{equation} Here, the elliptic integrals are defined by \begin{equation*}\label{eq:EK}
\begin{split}
	\calE_K(\alp) = \int_0^1 (1-t^2)^{-\frac12}(1-\alp t^2)^{-\frac12} \, \ud t, 
\end{split}
\end{equation*}
\begin{equation*}\label{eq:EE}
\begin{split}
	\calE_E(\alp) = \int_0^1 (1-t^2)^{-\frac12}(1-\alp t^2)^\frac12 \, \ud t 
\end{split}
\end{equation*} for $0 \le \alp < 1$. The integrals $\calE_K$ and $\calE_E$ are usually referred to as the \textit{complete elliptic integrals} of the first and second kind, respectively. We note the relations \begin{equation*}\label{eq:DE}
\begin{split}
	\calE_E'(\alp) = \frac{1}{2\alp}\left( \calE_E(\alp)- \calE_K(\alp) \right), \qquad \calE_K'(\alp) = \frac{1}{2\alp(1-\alp)}\left( \calE_E(\alp)- (1-\alp) \calE_K(\alp) \right),
\end{split}
\end{equation*} which can be used to verify \eqref{eq:DF} and \eqref{eq:D2F}. 

\subsection{Stream function and Biot--Savart law} 

The stream function for axisymmetric no-swirl flows has the form \begin{equation}\label{eq:psi-origin}
	\begin{split}
		\psi(r,z) =  \frac{1}{2\pi} \iint_{\Pi} \left[\int_0^{\pi} \frac{r\bar{r}\ct}{\sqrt{r^2-2r\bar{r}\ct + \bar{r}^2 + (z-\bar{z})^2}}  \ud\tht  \right] \omg(\bR,\bz) \, \ud \bR \ud \bz .  
	\end{split}
\end{equation}
Following the notation of Feng--\v{S}ver\'{a}k  \cite{FeSv}, we define \begin{equation*}
	\begin{split}
		\zeta(r,\bR,z,\bz) := \left(\frac{(r-\bR)^2 + (z-\bz)^2}{r\bR}\right)^{\frac12}. 
	\end{split}
\end{equation*} Then, one may perform the $\tht$-integration in \eqref{eq:psi-origin} and obtain a simple expression for $\psi$: \begin{equation*}\label{eq:psi}
\begin{split}
	\psi(r,z) = \iint_{\Pi} G(r,\bR,z,\bz)\omg(\bR,\bz) \, \ud \bR \ud \bz , 
\end{split}
\end{equation*} with \begin{equation*}\label{eq:G}
\begin{split}
	G(r,\bR,z,\bz):= \frac{(r\bR)^\frac12}{2\pi} \calF(\zeta^2(r,\bR,z,\bz)) . 
\end{split}
\end{equation*} With this stream function, the velocity is given by \begin{equation}\label{eq:u}
\begin{split}
	u^r (r,z)= - \frac{1}{r} \rd_{z} \psi (r,z), \qquad 	u^z (r,z)=  \frac{1}{r} \rd_{r} \psi (r,z). 
\end{split}
\end{equation} One may check directly with \eqref{eq:u} that \begin{equation*}
\begin{split}
	\rd_r (r u^r) + \rd_z (r u^z) = 0, \qquad \omg =  {-\rd_r u^z + \rd_z u^r}. 
\end{split}
\end{equation*} That is, $u = (u^r,u^z)$ indeed defines an incompressible vector field in $\bbR^3$ with associated curl $\omg$. Differentiating $G$, we obtain convenient integral expressions \begin{equation*}\label{eq:ur-G}
\begin{split}
	u^r(r,z) = \iint_{\Pi}  \frac{\bz-z}{\pi r^\frac32 \bR^\frac12} \calF'(\zeta^2)\omg(\bR,\bz) \, \ud \bR \ud \bz,  
\end{split}
\end{equation*}\begin{equation*}\label{eq:uz-G}
\begin{split}
	u^z(r,z) = \iint_{\Pi} \left(  \frac{r-\bR}{\pi r^\frac32 \bR^\frac12} \calF'(\zeta^2) + \frac{\bR^\frac12}{4\pi r^\frac32}\left[ \calF(\zeta^2)-2\zeta^2 \calF'(\zeta^2) \right] \right) \omg(\bR,\bz) \, \ud \bR \ud \bz. 
\end{split}
\end{equation*} We note that the
 energy of the flow is given simply by \begin{equation*}
\begin{split}
	E[\omega] =  {\pi}
	 \iint_{\Pi} \psi(r,z) \omg(r,z) \, \ud r \ud z =
	 {\frac 1 2}  \iiiint_{\Pi^{2}}  {(r\bR)^{\frac12}} \calF(\zt^2)  \, \omg(r,z) \omg(\bR,\bz) \, \ud r \ud z \ud \bR \ud \bz . 
\end{split}
\end{equation*} We remark that $0\le E<\infty $ for 
 {$ r^{-1}\omg \in L^1\cap L^\infty(\mathbb{R}^3)$ with
$r\omega\in  L^1(\mathbb{R}^3)$. In this case, the quantity $E[\omega]$ represents  the kinetic energy of the flow (after integration by parts):
$$
E=\frac 1 2 \iiint_{\mathbb{R}^3}|u|^2 \ud \bold{x}
$$  (\textit{e.g.} see Lemmas 2.3, 2.4 of \cite{Choi2020}). 
 }

\subsection{Axisymmetric flows without swirl} Let us briefly review the well-posedness theory for axisymmetric flows without swirl. A standard reference is \cite{MB}. 

\medskip \noindent \textbf{Global existence and uniqueness with $L^1\cap L^\infty$}. As we have mentioned earlier, there is a unique global-in-time solution to \eqref{eq:Euler-axisym-no-swirl} under the assumption $\omg_0, r^{-1}\omg_0 \in L^1 \cap L^\infty (\bbR^3)$. This is a natural extension of Yudovich theory for two-dimensional Euler. The analogy is apparent if one introduces the  {\textit{relative} vorticity} function $$\xi(t,r,z) = r^{-1}\omg(t,r,z):$$ \eqref{eq:Euler-axisym-no-swirl} is simply \begin{equation*}\label{eq:E-xi}
	\begin{split}
		\rd_t \xi + u\cdot\nb \xi = 0. 
	\end{split}
\end{equation*} Since the velocity is incompressible, formally we have $\nrm{\xi(t,\cdot)}_{L^p} = \nrm{\xi_0}_{L^p}$ for all $1\le p \le\infty$. Then, using the estimate (\textit{e.g.} see Lemma 2 of  \cite{Danchin}) \begin{equation*}
\begin{split}
	\nrm{r^{-1}  {u^r}}_{L^{\infty}} \le C \nrm{\xi}_{L^1\cap L^\infty} = C \nrm{\xi_0}_{L^1\cap L^\infty}, 
\end{split}
\end{equation*} we obtain a priori estimates for the $L^p$--norms of the vorticity: \begin{equation*}
\begin{split}
	\frac{d}{dt} \nrm{\omg}_{L^p} \le C_{p} \nrm{\xi_0}_{L^1\cap L^\infty} \nrm{\omg}_{L^p} ,
\end{split}
\end{equation*} \begin{equation*}
\begin{split}
	 \nrm{\omg(t,\cdot)}_{L^p}  \le  \nrm{\omg_0}_{L^p} \exp\left( Ct  \nrm{\xi_0}_{L^1\cap L^\infty} \right). 
\end{split}
\end{equation*} In particular, we see that under the assumption $\xi_0 \in L^1 \cap L^\infty$, higher regularity of the vorticity propagates in time, using the Beale--Kato--Majda criterion (\cite{BKM}). For global existence and uniqueness, the assumption $\omg_0, \xi_0 \in L^1 \cap L^\infty (\bbR^3)$ can be weakened; see \cite{Danchin} for instance. 

\medskip \noindent \textbf{Flow map}. Given a solution in the class $\omg, \xi \in L^\infty_{t,loc}( L^1 \cap L^\infty)(\bbR^3)$, we shall denote $\Phi(t,\cdot)$ to be the associated flow map, which is defined by the unique solution to the ODE \begin{equation}\label{ode}
	\begin{split}
		\frac{d}{dt}\Phi(t,\bold{x}) = u(t,\Phi(t,\bold{x})), \qquad \Phi(0,\bold{x})=\bold{x}
	\end{split}
\end{equation}  {(for solvability,   \textit{e.g.} see 
Section 2.4 of \cite{CJ_Hill}).}
Then, along the flow, we have the important Cauchy formula \begin{equation}\label{Cauchy}
\begin{split}
	\frac{\omg(t,\Phi(t,\bold{x}))}{\Phi^r(t,\bold{x})} = \frac{\omg_0(\bold{x})}{r}
\end{split}
\end{equation} which follows from the conservation of $\xi$ along $\Phi$. Here, $\Phi = (\Phi^r, \Phi^z)$ in cylindrical coordinates.

\medskip

\noindent \textbf{A priori estimate}.  {For later use, we recall a simple estimate. }

\begin{lemma}\label{lem:X-q} Define \begin{equation*}
		\begin{split}
			X(r,\bR,z,\bz,\tht) := r^2 + \bR^2 - 2r\bR \ct + (z-\bz)^2. 
		\end{split}
	\end{equation*}
	For any $q \in [0,3)$,
	 we have \begin{equation*}\label{eq:X-q}
		\begin{split}
			\sup_{ (\bR,\bz) \in \Pi }  \iint_{\Pi}  { \int_0^{2\pi} X^{-\frac{q}{2}} \ud\theta} |\omg(r,z)| \,   \ud r \ud z \le C_q\nrm{r^{-1}\omg}_{L^1(\bbR^3)}^{1-\frac{q}{3}} \nrm{r^{-1}\omg}_{L^{\infty}(\bbR^3)}^{\frac{q}{3}}  
		\end{split}
	\end{equation*} with a constant $C_q>0$ depending only on $q$. 
\end{lemma}
\begin{proof}
	We note that for $\bold{x} = (r,\tht,z)$ and $\bar{\bold{x}} = (\bR,0,\bz)$, we have \begin{equation*}
		\begin{split}
			\iint_{\Pi}  { \int_0^{2\pi} X^{-\frac{q}{2}} \ud \theta}|\omg(r,z)| \, \ud r \ud z =  \iiint_{\bbR^3} \frac{1}{|\bar{\bold{x}}-\bold{x}|^{q}} \frac{|\omg(\bold{x})|}{r} \, \ud \bold{x}.
		\end{split}
	\end{equation*} The statement is clear when $q = 0$. For $0<q<3$, we consider the regions $\{ |\bold{x}-\bar{\bold{x}}| \le L \}$ and $\{ |\bold{x}-\bar{\bold{x}}| > L \}$, for some $L>0$ to be determined. In the latter region, we have that \begin{equation*}
		\begin{split}
			\iiint_{|\bold{x}-\bar{\bold{x}}|>L} \frac{1}{|\bar{\bold{x}}-\bold{x}|^{q}} \frac{|\omg(\bold{x})|}{r} \, \ud \bold{x} \le CL^{-q}\nrm{r^{-1}\omg}_{L^1} 
		\end{split}
	\end{equation*} since $q>0$ and in the former, we obtain \begin{equation*}
		\begin{split}
			\iiint_{|\bold{x}-\bar{\bold{x}}|\le L} \frac{1}{|\bar{\bold{x}}-\bold{x}|^{q}} \frac{|\omg(\bold{x})|}{r} \, \ud \bold{x} \le \nrm{r^{-1}\omg}_{L^{\infty}} \iiint_{|\bold{x}|\le L}|\bold{x}|^{-q}\, \ud \bold{x} \le C_q L^{3-q}\nrm{r^{-1}\omg}_{L^{\infty}} 
		\end{split}
	\end{equation*} since $q<3$. The choice \begin{equation*}
		\begin{split}
			L = \left(  \frac{ \nrm{r^{-1}\omg}_{L^{1} }}{ \nrm{r^{-1}\omg}_{L^{\infty}} }\right)^{\frac13}
		\end{split}
	\end{equation*} gives the desired estimate. 
\end{proof}

\medskip \noindent \textbf{The $t^2$ bound on the vorticity maximum}. 
We provide a proof that the vorticity maximum for general axisymmetric flows without swirl cannot grow faster than $t^2$ in time, assuming that the initial vorticity 
 { satisfies 
$$	
			 \sup\{ r : (r,z) \in \supp(\omg_0(\cdot)) \} 	
			<\infty $$
together with   \eqref{init_assum}.
}
  We need the following lemma  {from Feng--\v{S}ver\'{a}k \cite{FeSv}.}
\begin{lemma}[{{\cite[Propositions 2.11,  {2.13}]{FeSv}}}] \label{lem:FS}
	The velocity $u = K[\omg]$ satisfies \begin{equation*}
		\begin{split}
			\nrm{u}_{L^\infty} \le C\nrm{r\omg}_{L^1}^{\frac14}\nrm{r^{-1}\omg}_{L^1}^{\frac14}\nrm{r^{-1}\omg}_{L^\infty}^{\frac12}. 
		\end{split}
	\end{equation*}
	\medskip
\end{lemma} Recall our convention that the $L^p$ norms are taken always with respect to the Lebesgue measure on $\bbR^3$. Now to prove the claim, we define $R(t) = \sup\{ r : (r,z) \in \supp(\omg(t,\cdot)) \}$. We may assume that $R(t)$ is increasing in time by redefining it to be $\sup_{t'\le t} R(t')$. Then \begin{equation*}
		\begin{split}
			\frac{d}{dt} R(t) &\le \nrm{u^r}_{L^\infty} \le  C\nrm{r\omg}_{L^1}^{\frac14}\nrm{r^{-1}\omg}_{L^1}^{\frac14}\nrm{r^{-1}\omg}_{L^\infty}^{\frac12} \\
			& \le CR(t)^\frac12 \nrm{ r^{-1}\omg}_{L^1}^{\frac12} \nrm{r^{-1}\omg}_{L^\infty}^{\frac12} \\
			& = CR(t)^\frac12 \nrm{ r^{-1}\omg_0}_{L^1}^{\frac12} \nrm{r^{-1}\omg_0}_{L^\infty}^{\frac12} . 
		\end{split}
	\end{equation*} Integrating in time gives \begin{equation}\label{est_R_t_2}
	\begin{split}
		R(t) \le C(1+t)^2. 
	\end{split}
\end{equation} Recalling the Cauchy formula \eqref{Cauchy}, the $t^2$--bound on the support can be translated to a bound on $\nrm{\omg(t,\cdot)}_{L^\infty}$: \begin{equation*}
\begin{split}
	\nrm{\omg(t,\cdot)}_{L^\infty} \le R(t) \nrm{r^{-1}\omg_0}_{L^\infty} \le C(1+t)^2,
\end{split}
\end{equation*}  {where $C>0$ depends on the initial data $\omega_0$.}

\section{Key Inequalities} 

From now on, we shall impose the assumption \eqref{eq:vort-assumption} on $\omg$, so that the dynamics reduces to $\Pi_+ = \{ (r,z) : r,z \ge 0 \}$. 

\subsection{Relation between $P$ and vorticity $L^p$ norms} 
We define
\begin{equation*}\label{eq:P}
	\begin{split}
		P(t) := 			\iint_{\Pi_+} -r^2 \omg(t,r,z) \, \ud r \ud z  = \iiint_{\bbR^3_{+}} -r^2\xi(t,\cdot) \, \ud x   \ge 0. 
	\end{split}
\end{equation*}
{
\begin{lemma}\label{lem:P-ineq-Linfty2} Let $0\leq \delta<1$. Assume that $ \nrm{\xi_0^{-1} \mathbf{1}_{\{\xi_0<0\}} }_{L^{\frac{1-\delta}{1-((2-\delta)/p)}} (\bbR^3_+)}< \infty$ for some $p\in[2-\delta,\infty]$
and
$R_0 = \sup\{ r : (r,z) \in \supp(\omg_0(\cdot)) \}<\infty$. Then, we have  
	\begin{equation*}\label{eq:P-ineq-Linfty2}
		\begin{split}
			P(t) \le R(t)^\delta \nrm{\omg(t,\cdot)}_{L^p (\bbR^3_+)}^{2-\delta} \nrm{\xi_0^{-1} \mathbf{1}_{\{\xi_0<0\}} }^{1-\delta}_{L^{\frac{1-\delta}{1-((2-\delta)/p)}} (\bbR^3_+)},
		\end{split}
	\end{equation*} where $R(t)=  \sup\{ r : (r,z) \in \supp(\omg(t,\cdot)) \}$.
\end{lemma}
\begin{proof}
	We estimate using H\"older's inequality \begin{equation*}
		\begin{split}
			\iiint_{\bbR^3_{+}} -r^2\xi \, \ud x &= \iiint_{\bbR^3_{+}} -r^\delta\frac{r^{1-\delta}}{\omg^{1-\delta}} \omg^{2-\delta} \, \ud x \\& \le R(t)^{\delta}\nrm{\omg(t,\cdot)}_{L^p(\bbR^3_+)}^{2-\delta} \left( \iiint_{ \{ \xi(t,\cdot) < 0 \} } \xi(t,\cdot)^{-\frac{1-\delta}{1-((2-\delta)/p)}} \, \ud x \right)^{1-\frac{(2-\delta)}{p}}
			\\			&=  R(t)^{\delta}  \nrm{\omg(t,\cdot)}_{L^p(\bbR^3_+)}^{2-\delta}\nrm{\xi_0^{-1} \mathbf{1}_{\{\xi_0<0\}} }^{1-\delta}_{L^{\frac{1-\delta}{1-((2-\delta)/p)}} (\bbR^3_+)}.
		\end{split}
	\end{equation*} 
	\noindent In the last equality, we have used that the distribution function of $\xi(t,\cdot)$ is invariant in time. 
\end{proof}
}

It is worth considering the special case in which $\xi_{0} = -\lmb \mathbf{1}_{\Omg_0}$ on $\bbR^3_+$ for some $\lmb>0$ and a bounded open set $\Omg_0 \subset \Pi_+ \cap \{ r \ge \frac12 \}$. The associated solution takes the form \begin{equation*}\label{eq:xi-patch}
	\begin{split}
		\xi(t,\cdot) = -\lmb \mathbf{1}_{\Omg(t)}.
	\end{split}
\end{equation*} Denoting $R(t) := \sup\{ r : \mbox{there exists } (r,z) \in \Omg(t) \},$ we immediately obtain from $|\Omg(t)|=|\Omg_0|$ ($|\cdot|$ denotes the three-dimensional Lebesgue measure) that \begin{equation*}
\begin{split}
	P(t) \le \lmb |\Omg_0| R^2(t),
\end{split}
\end{equation*} and combining this with \begin{equation*}
\begin{split}
	\nrm{\omg(t,\cdot)}_{L^\infty} \ge  { \frac12\lmb R(t), }
\end{split}
\end{equation*} we arrive at the following precise result: \begin{equation*}\label{eq:P-patch}
		\begin{split}
			\nrm{\omg(t,\cdot)}_{L^{\infty}} \ge \frac12\left( |\Omg_0|^{-1} \nrm{\xi_0}_{L^\infty} P(t)  \right)^{\frac{1}{2}}. 
		\end{split}
	\end{equation*}  

\subsection{Estimates on kinetic energy and derivative of $P$}

In addition to  \begin{equation*}
	\begin{split}
		X(r,\bR,z,\bz,\tht) = r^2 + \bR^2 - 2r\bR \ct + (z-\bz)^2
	\end{split}
\end{equation*} which  has been defined in the above, we set \begin{equation*}
\begin{split}
			\bar{X}(r,\bR,z,\bz,\tht) &= r^2 + \bR^2 + 2r\bR \ct + (z-\bz)^2, \\
			Y(r,\bR,z,\bz,\tht) &= r^2 + \bR^2 - 2r\bR \ct + (z+\bz)^2, \\
			\bar{Y}(r,\bR,z,\bz,\tht) & = r^2 + \bR^2 + 2r\bR \ct + (z+\bz)^2. 
\end{split}
\end{equation*} 
For ease of notation, let us write $\omg = \omg(r,z)\,\ud r \ud z$ and $\bar\omg = \omg(\bR,\bz)\,\ud \bR \ud \bz$. 
\begin{lemma}\label{lem:E-P-bounds}
	The kinetic energy $E \ge 0 $ satisfies \begin{equation}\label{eq:E-bounds}
		\begin{split}
			\frac{E}{C}\le \iiiint_{(\Pi_+)^{2}}\left[ \int_0^{\frac{\pi}{2}} { \frac{(r \bR)^2 (z\bz) (\ct)^2}{X^\frac12 \bX Y }   } \, \ud\tht\right] \, \omg \bar{\omg} \, \le CE. 
		\end{split}
	\end{equation} Moreover, $\dot{P}:= \frac{d}{dt} P \ge 0 $ satisfies \begin{equation}\label{eq:P-prime-bounds}
	\begin{split}
		\frac{\dot{P}}{C} \le \iiiint_{(\Pi_+)^{2}}\left[ \int_0^{\frac{\pi}{2}}\frac{(r \bR)^2 (z+\bz) (\ct)^2}{Y^\frac32\bY }  \, \ud\tht\right] \, \omg \bar{\omg} \, \le C \dot{P}. 
	\end{split}
\end{equation} Here, $C>0$ is an absolute constant. 
\end{lemma}
 
\begin{proof}
	We begin with the expression \begin{equation*}
		\begin{split}
			E & =  {\frac{1}{2}} \iiiint_{\Pi^{2}}   \left[\int_0^{\pi} \frac{r\bar{r}\ct}{X^{\frac12}} \, \ud\tht  \right] \omg \bar{\omg} =   {\frac{1}{2}} \iiiint_{\Pi^{2}}   \left[\int_0^{\frac{\pi}{2}} r\bar{r}\ct\left(\frac{1}{X^{\frac12}} -\frac{1}{\bX^{\frac12}} \right)\, \ud\tht  \right] \omg \bar{\omg} \\
			& =  {2}\iiiint_{\Pi^{2}}   \left[\int_0^{\frac{\pi}{2}}  \frac{(r\bar{r}\ct)^2}{X^{\frac12}\bX^{\frac12}(X^{\frac12}+\bX^{\frac12})}  \, \ud\tht  \right] \omg \bar{\omg} .
		\end{split}
	\end{equation*} Next, using odd symmetry of $\omg$ and $\bar{\omg}$ in $z$, \begin{equation*}
	\begin{split}
		E = {4}\iiiint_{(\Pi_+)^{2}}   \left[\int_0^{\frac{\pi}{2}} (r\bar{r}\ct)^2  \left(\frac{1}{X^{\frac12}\bX^{\frac12}(X^{\frac12}+\bX^{\frac12})} - \frac{1}{Y^{\frac12}\bY^{\frac12}(Y^{\frac12}+\bY^{\frac12})} \right) \, \ud\tht  \right] \omg \bar{\omg} .
	\end{split}
\end{equation*} We compute \begin{equation*}
\begin{split}
	&\frac{1}{X^{\frac12}\bX^{\frac12}(X^{\frac12}+\bX^{\frac12})} - \frac{1}{Y^{\frac12}\bY^{\frac12}(Y^{\frac12}+\bY^{\frac12})}  \\
	&\qquad = \frac{ (Y\bY)^{\frac12} (Y^{\frac12}+\bY^{\frac12} - X^{\frac12}-\bX^{\frac12} ) + ( (Y\bY)^{\frac12} -(X\bX)^{\frac12} )( X^{\frac12}+\bX^{\frac12} ) }{X^{\frac12}\bX^{\frac12}(X^{\frac12}+\bX^{\frac12})Y^{\frac12}\bY^{\frac12}(Y^{\frac12}+\bY^{\frac12})} =: I + II. 
\end{split}
\end{equation*} We consider \begin{equation*}
\begin{split}
	II = \frac{(Y\bY)^{\frac12} -(X\bX)^{\frac12}  }{X^{\frac12}\bX^{\frac12} Y^{\frac12}\bY^{\frac12}(Y^{\frac12}+\bY^{\frac12})} = \frac{Y\bY -X\bX }{X^{\frac12}\bX^{\frac12} Y^{\frac12}\bY^{\frac12}(Y^{\frac12}+\bY^{\frac12})((Y\bY)^{\frac12} +(X\bX)^{\frac12} )}.
\end{split}
\end{equation*} Noting that \begin{equation*}
\begin{split}
	Y\bY -X\bX = 2z\bz ( X + \bX + Y + \bY)
\end{split}
\end{equation*} and \begin{equation*}
\begin{split}
	X \le \bX, \qquad Y \le \bY, \qquad X \le Y, \qquad \bX \le \bY, 
\end{split}
\end{equation*} we observe that \begin{equation*}
\begin{split}
	Y\bY -X\bX  \sim z\bz\bY, \qquad X^{\frac12}\bX^{\frac12} Y^{\frac12}\bY^{\frac12}(Y^{\frac12}+\bY^{\frac12})((Y\bY)^{\frac12} +(X\bX)^{\frac12} ) \sim X^{\frac12}\bX^{\frac12} Y^{\frac12}\bY^{\frac12}\bY^{\frac12}(Y\bY)^{\frac12}.
\end{split}
\end{equation*} Here, we write $A \sim B$ if there is an absolute constant $C>0$ such that $C^{-1}A \le B \le CA$. This gives \begin{equation*}
\begin{split}
	II \sim \frac{z\bz\bY}{ X^{\frac12}\bX^{\frac12} Y^{\frac12}\bY^{\frac12}\bY^{\frac12}(Y\bY)^{\frac12} } \sim  \frac{z\bz }{ X^{\frac12}\bX^{\frac12} Y^{\frac12}(Y\bY)^{\frac12} }.
\end{split}
\end{equation*} Next, \begin{equation*}
\begin{split}
	I = \frac{ Y^{\frac12}+\bY^{\frac12} - X^{\frac12}-\bX^{\frac12} }{X^{\frac12}\bX^{\frac12}(X^{\frac12}+\bX^{\frac12})(Y^{\frac12}+\bY^{\frac12}) } =: I_1+I_2
\end{split}
\end{equation*} with \begin{equation*}
\begin{split}
	I_1 =  \frac{ Y^{\frac12} - X^{\frac12}  }{X^{\frac12}\bX^{\frac12}(X^{\frac12}+\bX^{\frac12})(Y^{\frac12}+\bY^{\frac12}) } = \frac{4z\bz }{X^{\frac12}\bX^{\frac12}(X^{\frac12}+\bX^{\frac12})(Y^{\frac12}+\bY^{\frac12})( Y^{\frac12} + X^{\frac12} ) }  \sim  \frac{z\bz }{X^{\frac12}\bX^{\frac12}\bX^{\frac12}Y^{\frac12}\bY^{\frac12}   } 
\end{split}
\end{equation*} Similarly, \begin{equation*}
\begin{split}
	I_2 \sim  \frac{z\bz }{X^{\frac12}\bX^{\frac12}\bX^{\frac12}\bY^{\frac12}\bY^{\frac12}   } \lesssim I_1. 
\end{split}
\end{equation*} Since \begin{equation*}
\begin{split}
	 \frac{z\bz }{X^{\frac12}\bX^{\frac12}\bX^{\frac12}Y^{\frac12}\bY^{\frac12}   }  +  \frac{z\bz }{ X^{\frac12}\bX^{\frac12} Y^{\frac12}(Y\bY)^{\frac12} } \sim \frac{z\bz}{X^\frac12\bX Y},
\end{split}
\end{equation*} we obtain \eqref{eq:E-bounds}.

\medskip

\noindent Next, we compute \begin{equation*}
	\begin{split}
		\dot{P} & = -\iint_{\Pi_+} 2r u^r (r,z) \omg (r,z) \, \ud r \ud z \\
		& = -\frac{1}{\pi}\iint_{\Pi} \iint_{\Pi_+}   \left[ \int_0^\pi \frac{(z-\bar{z})r \bar{r}\ct}{X^{\frac{3}{2}}} \ud \tht \right]  \omg \bar{\omg} \\
		& = \frac{1}{\pi} \iiiint_{(\Pi_+)^2}  \left[ \int_0^{\pi}  \frac{(z+\bz)r\bR\ct}{Y^\frac32} \, \ud \tht \right] \omg \bar{\omg} \\
		&  = \frac{4}{\pi} \iiiint_{(\Pi_+)^2}  \left[ \int_0^{\frac{\pi}{2}}  (z+\bz)(r\bR)^2 (\ct)^2 \frac{  \bY^2 + \bY Y+Y^2 }{(Y\bY)^\frac32(Y^\frac32+\bY^\frac32)}  \, \ud \tht \right] \omg \bar{\omg}. 
	\end{split}
\end{equation*} Then, we simply note that 
\begin{equation*}
	\begin{split}
		\bY^2 + \bY Y+Y^2 \sim \bY^2, \qquad (Y\bY)^\frac32(Y^\frac32+\bY^\frac32) \sim Y^{\frac32} \bY^3
	\end{split}
\end{equation*} to derive \begin{equation*}
	\begin{split}
			\dot{P} \sim \iiiint_{(\Pi_+)^2}  \left[ \int_0^{\frac{\pi}{2}}  \frac{ (z+\bz)(r\bR)^2 (\ct)^2 }{ Y^{\frac32}\bY }  \, \ud \tht \right] \omg \bar{\omg}. 
	\end{split}
\end{equation*} This finishes the proof. 
\end{proof}

\subsection{Monotonicity Lemma}
For simplicity, we shall denote \begin{equation*}
	\begin{split}
		Z(t) :=  \iint_{\Pi_+}  -z \omg (t,r,z) \, \ud r \ud z . 
	\end{split}
\end{equation*} We have already seen that \begin{equation*}
\begin{split}
	 P(t) = \iint_{\Pi_+} - r^2 \omg(t,r,z) \, \ud r \ud z 
\end{split}
\end{equation*} is monotone increasing with $t$. In the meanwhile, we note that $0\le Z, P < +\infty$ for $t>0$ if they are finite at $t = 0$. 
\begin{lemma}\label{lem:z-down}
	We have \begin{equation}\label{eq:z-down}
		\begin{split}
			\dot{Z} := \frac{d}{dt} \iint_{\Pi_+}  -z \omg (t,r,z) \, \ud r \ud z  {< 0.}
		\end{split}
	\end{equation} 
\end{lemma}  

\begin{proof}[Proof of Lemma \ref{lem:z-down}]
	We begin with \begin{equation*}
		\begin{split}
			\frac{d}{dt} \iint_{\Pi_+} z \omg  \, \ud r \ud z = \iint_{\Pi_+}  u^z\omg \, \ud r \ud z . 
		\end{split}
	\end{equation*}
 Symmetrizing the kernel of $u^z$ in $(r,z)\leftrightarrow(\bR,\bz)$, we obtain with \begin{equation*}
\begin{split}
	 \calH[r,\bR,z,\bz]:=-\frac{(r-\bR)^2}{2\pi r^\frac32 \bR^\frac32} \calF'(\zeta^2) + \frac{\bR^2+r^2}{ {8}\pi r^\frac32\bR^\frac32}\left[ \calF(\zeta^2)-2\zeta^2 \calF'(\zeta^2) \right] 
\end{split}
\end{equation*} that \begin{equation*}
\begin{split}
	\iint_{\Pi_+}  u^z\omg \, \ud r \ud z  = \iint_{\Pi_+} \iint_{\Pi_+} \left[\calH[r,\bR,z,\bz] - \calH[r,\bR,z,-\bz]\right] \omega(r,z)\omega(\bR,\bz) \,  \ud r \ud z  \ud \bR \ud \bz .
\end{split}
\end{equation*} We need to show that $\calH[r,\bR,z,\bz] - \calH[r,\bR,z,-\bz] {>0}$, but since the dependence of $\calH$ in $z, \bz$ is only through $(z-\bz)^2$ in $\zeta^2$, it suffices to show that $-\calF'(s)$ and $\calF(s)-2s\calF'(s)$ are  {strictly} \textit{decreasing} functions of $s$ for all $s\in\bbR_+$. 

\medskip

\noindent (i) $\calF'(s)$. We need to check  {$\calF''>0$.} However, this is clear from the integral representation in \eqref{eq:D2F} after symmetrizing $\tht$ with $\pi-\tht$. 

\medskip

\noindent (ii) $\calF(s)-2s\calF'(s)$. We need to prove that \begin{equation*}
	\begin{split}
		(\calF(s)-2s\calF'(s))' = - \calF'(s) - 2s\calF''(s)  {< 0.} 
	\end{split}
\end{equation*} Using \eqref{eq:DF} and \eqref{eq:D2F}, we explicitly check that \begin{equation*}
\begin{split}
	 \calF'(s) + 2s\calF''(s) = \frac{1}{s(4+s)^\frac32} \left(  s \calE_K\left( \frac{4}{4+s}  \right)+ (4-s)\calE_E\left( \frac{4}{4+s}  \right) \right)
\end{split}
\end{equation*} but \begin{equation*}
\begin{split}
	 s \calE_K\left( \frac{4}{4+s}  \right)+ (4-s)\calE_E\left( \frac{4}{4+s}  \right)= 4\calE_E\left( \frac{4}{4+s}  \right) + s\left(  \calE_K\left( \frac{4}{4+s}  \right)- \calE_E\left( \frac{4}{4+s}  \right) \right)  {> 0} 
\end{split}
\end{equation*} since $\calE_E {>0}$ and  {$\calE_K>\calE_E$.} This finishes the proof of \eqref{eq:z-down}. 
\end{proof} 

\section{Infinite vortex stretching}\label{sec_stretching}

\subsection{Preliminary growth}

We are now in a position to state and prove the first key proposition. 
\begin{proposition}\label{prop:key-1} We have the following lower bound for all $t\ge 0$: 
	\begin{equation*}\label{eq:P-lb}
		\begin{split}
				P(t) \ge \left(  	C_q(E_0 X_0^{-1})^{\frac{q}{q-1}} t + 	P_0^{\frac{q}{q-1}}      \right)^{\frac{q-1}{q}}, \qquad 1 < q < \frac{15}{13},
		\end{split}
	\end{equation*} where $C_q>0$ is a constant depending only on $q$, $E_0$ is the kinetic energy, and \begin{equation*}
	\begin{split}
		X_0 = \left( \nrm{\xi_0}_{L^1}^{4-\frac{10}{3}q}   \nrm{\xi_0}_{L^\infty}^{\frac{q}{3}}Z_0^{4(q-1)}P_0^{1-q}  +  \nrm{\xi_0}_{L^1}^{3-\frac{7}{3}q} \nrm{\xi_0}_{L^\infty}^{\frac{q}{3}} Z_0^{2(q-1)} \right)^{\frac{1}{q}}
	\end{split}
\end{equation*} is a constant depending only on the initial data. 
\end{proposition}
 Note that the statements of Theorem \ref{thm:main} follows from the above proposition. 
   {Then, Corollary \ref{cor:main2}   follows from Lemma \ref{lem:P-ineq-Linfty2} thanks to $t^2$--bound \eqref{est_R_t_2} on support.}

\begin{proof} To begin with, we recall the bounds 
\eqref{eq:E-bounds}, \eqref{eq:P-prime-bounds}
on $E$ and $P$ from Lemma \ref{lem:E-P-bounds}. Then, we estimate using H\"older's inequality \begin{equation}\label{en_est}
		\begin{split}
			E &\lesssim \iiiint_{(\Pi_+)^{2}} \int_0^{\frac{\pi}{2}} \frac{(r \bR)^2 (z\bz) (\ct)^2}{ X^\frac12 \bX Y } \, \ud\tht \,  \omg \bar{\omg} \,  \\
			&\lesssim \iiiint_{(\Pi_+)^2} \int_0^{\frac{\pi}{2}} \frac{(r \bR)^2 (z+\bz) (\ct)^2}{ X^{\frac12} \bX Y^{\frac12}  }  \, \ud\tht \,   \omg \bar{\omg} \\
			&\lesssim (\dot{P})^{1-\frac{1}{q}} J^{\frac{1}{q}}
		\end{split}
	\end{equation} for some $1<q \le \frac32$ to be determined below, where $J$ is given by  \begin{equation*}
		\begin{split}
			J &= \iiiint_{(\Pi_+)^2} \int_0^{\frac{\pi}{2}} \frac{(r \bR)^2 (z+\bz) \bY^{q-1} (\ct)^2}{(X\bX^2)^{\frac{q}{2}} Y^{\frac{3}{2}-q}    } \, \ud\tht \, \omg \bar{\omg} \\
			& = \iiiint_{(\Pi_+)^2} \int_0^{\frac{\pi}{2}} \frac{2(r \bR)^2 \bz \bY^{q-1} (\ct)^2}{(X\bX^2)^{\frac{q}{2}} Y^{\frac{3}{2}-q}    } \, \ud\tht \, \omg \bar{\omg}.
		\end{split}
	\end{equation*} We have used the lower bound of $P$ from Lemma \ref{lem:E-P-bounds}. We now estimate 
\begin{equation}\label{jj1j2}
J \le C( J_1 + J_2 )
\end{equation}	
	 using \begin{equation*}
		\begin{split}
			\bY^{q-1} \le C(\bX^{q-1} + Y^{q-1}) 
		\end{split}
	\end{equation*} where \begin{equation}\label{j1j2}
\begin{split}
	J_1 = \iiiint_{(\Pi_+)^2} \int_0^{\frac{\pi}{2}} \frac{ (r \bR)^2 \bz \bX^{q-1} (\ct)^2}{(X\bX^2)^{\frac{q}{2}} Y^{\frac{3}{2}-q}    } \, \ud\tht \, \omg \bar{\omg},  \qquad J_2= \iiiint_{(\Pi_+)^2} \int_0^{\frac{\pi}{2}} \frac{ (r \bR)^2 \bz Y^{q-1} (\ct)^2}{(X\bX^2)^{\frac{q}{2}} Y^{\frac{3}{2}-q}    } \, \ud\tht \, \omg \bar{\omg}.
\end{split}
\end{equation}

\medskip

\noindent \textit{Estimate of $J_1$}. We begin with rewriting it as \begin{equation*}
\begin{split}
	J_1 = 2 \iint_{\Pi_+} r^2 |\omg(r,z)| Q_1(r,z) \, \ud r \ud z ,  
\end{split}
\end{equation*} where \begin{equation*}
\begin{split}
	Q_1(r,z) & :=  \iint_{\Pi_+}\int_0^{\frac{\pi}{2}} \frac{ \bR^2 \bz  (\ct)^2}{ X^{\frac{q}{2}}\bX Y^{\frac{3}{2}-q}    }  \, \ud\tht \, |\bar{\omg}| \\
	& \lesssim \iint_{\Pi_+} \int_0^{\frac{\pi}{2}} \frac{  \bz^{2(q-1)}    (\ct)^2 }{ X^{\frac{q}{2}}    }  \, \ud\tht \, |\bar{\omg}|  \\
	& \lesssim \left(\iint_{\Pi_+} \int_0^{\frac{\pi}{2}}  \frac{1}{X^{\frac{q}{2(3-2q)}} } \, \ud \tht \, |\bar{\omg}| \right)^{3-2q} \left( \iint_{\Pi_+} \bz |\bar{\omg} | \right)^{2(q-1)} .
\end{split}
\end{equation*} When $q < \frac{9}{7}$, we obtain that \begin{equation*}
\begin{split}
	\nrm{Q_1(r,z)}_{L^\infty(\Pi_+)} \le C \nrm{\xi_0}_{L^1}^{3-\frac{7}{3}q} \nrm{\xi_0}_{L^\infty}^{\frac{q}{3}} Z_0^{2(q-1)}, \qquad J_1 \le  C  \nrm{\xi_0}_{L^1}^{3-\frac{7}{3}q} \nrm{\xi_0}_{L^\infty}^{\frac{q}{3}} Z_0^{2(q-1)} P. 
\end{split}
\end{equation*} 

\medskip

\noindent \textit{Estimate of $J_2$}. We first bound \begin{equation*}
	\begin{split}
		J_2 &= \iiiint_{(\Pi_+)^2} \int_0^{\frac{\pi}{2}} \frac{ (r \bR)^2 \bz Y^{q-1} (\ct)^2}{(X\bX^2)^{\frac{q}{2}} Y^{\frac{3}{2}-q}    } \, \ud\tht \, \omg \bar{\omg} \\
		&\lesssim \iiiint_{(\Pi_+)^2} \int_0^{\frac{\pi}{2}} \frac{ r^{4-2q} \bz^{4(q-1)} (\ct)^2}{X^{\frac{q}{2}} } \, \ud\tht \, \omg \bar{\omg} .
	\end{split}
\end{equation*}  Let us now estimate \begin{equation*}
\begin{split}
	Q_2(r,z) & := \iint_{\Pi_+}\int_0^{\frac{\pi}{2}} \frac{ \bz^{4(q-1)} (\ct)^2}{X^{\frac{q}{2}} } \, \ud\tht \,|\bar{\omg}| \\
	& \lesssim \left(\iint_{\Pi_+} \int_0^{\frac{\pi}{2}}  \frac{1}{X^{\frac{q}{2(5-4q)}} } \, \ud \tht \, |\bar{\omg}| \right)^{5-4q} \left( \iint_{\Pi_+} \bz |\bar{\omg} | \right)^{4(q-1)} \\
	& \lesssim  \nrm{\xi_0}_{L^1}^{5-\frac{13}{3}q} \nrm{\xi_0}_{L^\infty}^{\frac{q}{3}}Z_0^{4(q-1)}
\end{split}
\end{equation*} This requires $q < \frac{15}{13}$. Then using H\"older's inequality \begin{equation*}
\begin{split}
	J_2  &\lesssim \iint_{\Pi_+} r^{4-2q} |\omg|^{2-q} \,  Q_2(r,z) |\omg|^{q-1} \,      \ud r\ud z \\
	 & \lesssim  \nrm{Q_2}_{L^\infty}  P^{2-q} \nrm{\xi_0}_{L^1}^{q-1} \\
	 & \lesssim  \nrm{\xi_0}_{L^1}^{5-\frac{13}{3}q}   \nrm{\xi_0}_{L^\infty}^{\frac{q}{3}}Z_0^{4(q-1)} \nrm{\xi_0}_{L^1}^{q-1} P^{2-q}  \\
	 & \lesssim \nrm{\xi_0}_{L^1}^{4-\frac{10}{3}q}   \nrm{\xi_0}_{L^\infty}^{\frac{q}{3}}Z_0^{4(q-1)}P_0^{1-q} P. 
\end{split}
\end{equation*}   Combining the estimates for $J_1$ and $J_2$, we arrive at \begin{equation*}
\begin{split}
	E \le C X_0 (\dot{P})^{1-\frac{1}{q}} P^{\frac{1}{q}}, \qquad 1 < q < \frac{15}{13}. 
\end{split}
\end{equation*} Here, we recall that \begin{equation*}
\begin{split}
	X_0 = \left( \nrm{\xi_0}_{L^1}^{4-\frac{10}{3}q}   \nrm{\xi_0}_{L^\infty}^{\frac{q}{3}}Z_0^{4(q-1)}P_0^{1-q}  +  \nrm{\xi_0}_{L^1}^{3-\frac{7}{3}q} \nrm{\xi_0}_{L^\infty}^{\frac{q}{3}} Z_0^{2(q-1)} \right)^{\frac{1}{q}}. 
\end{split}
\end{equation*} Therefore, \begin{equation*}
\begin{split}
	E_0^{\frac{q}{q-1}} \le C 	X_0^{\frac{q}{q-1}}  \frac{d}{dt} \left( P^{\frac{q}{q-1}} \right),
\end{split}
\end{equation*} or \begin{equation*}
\begin{split}
	P(t) \ge \left(  	C(E_0 X_0^{-1})^{\frac{q}{q-1}} t + 	P_0^{\frac{q}{q-1}}      \right)^{\frac{q-1}{q}}. 
\end{split}
\end{equation*} This finishes the proof. 
\end{proof}

\subsection{Enhancing the growth rate}

In this section, we complete the proof of Theorem \ref{thm:main3}. We shall need a few additional lemmas. 

\begin{lemma}\label{lem:F-est}
	For any $0<\tau\le\frac{3}{2}$, we have \begin{equation}\label{eq:F-est}
		\begin{split}
			\int_0^{\frac{\pi}{2}} \frac{(\ct)^2}{X^{\frac12} \bX} \, \ud \tht \le C_\tau \frac{(r\bR)^{\tau-\frac32}}{d^{2\tau}},\qquad d:= \sqrt{(r-\bR)^2 + (z-\bz)^2}. 
		\end{split}
	\end{equation}
\end{lemma}
\begin{proof}
	This comes from the following estimate of $\calF$ (see \cite[Lemma 2.7]{FS}) \begin{equation*}
		\begin{split}
			\calF(s) \lesssim_\tau s^{-\tau}, \qquad \tau\in(0,3/2]. 
		\end{split}
	\end{equation*} We plug in $s = d^2/(r\bR)$ to obtain \begin{equation*}
	\begin{split}
		\sqrt{r\bR} \int_0^\pi \frac{\ct}{X^{\frac12}}\,\ud \tht \lesssim_\tau \frac{(r\bR)^\tau}{d^{2\tau}}. 
	\end{split}
\end{equation*} Symmetrizing the left hand side, \begin{equation*}
\begin{split}
	\sqrt{r\bR} \int_0^\pi \frac{\ct}{X^{\frac12}}\,\ud \tht = \sqrt{r\bR} \int_0^{\frac{\pi}{2}} \left( \frac{\ct}{X^{\frac12}} - \frac{\ct}{\bX^{\frac12}} \right) \,\ud \tht = 4(r\bR)^{\frac32}	\int_0^{\frac{\pi}{2}} \frac{(\ct)^2}{X^{\frac12} \bX} \, \ud \tht . 
\end{split}
\end{equation*} This gives the lemma. 
\end{proof}

\begin{lemma}\label{lem:2D} For any $0\le \alp < 2$, we have 
	\begin{equation*}\label{eq:2D}
		\begin{split}
			\sup_{(\bR,\bz)}\iint_{\Pi_+} \frac{|\omg|}{d^\alp} \, \ud r \ud z \le C_\alp \nrm{\omg}_{L^1(\Pi_+)}^{1-\frac{\alp}{2}} \nrm{\omg}_{L^\infty(\Pi_+)}^{\frac{\alp}{2}}. 
		\end{split}
	\end{equation*}
\end{lemma}
\begin{proof}
	The proof is completely parallel to that for Lemma \ref{lem:X-q} in the above. See also Lemma 2.1. of \cite{ISG99}. 
\end{proof}

\begin{proof}[Proof of Theorem \ref{thm:main3}]
	Towards a contradiction, we assume that there exists $M>0$ such that \begin{equation*}
		\begin{split}
			\sup_{t\ge1}\frac{\nrm{\omg(t,\cdot)}_{L^\infty}}{t^{\bt}}\le M. 
		\end{split}
	\end{equation*} Based on this assumption, we now obtain \textit{improved estimates} on $J_1$ and $J_2$ (in \eqref{j1j2}). \textbf{In this proof, implicit constants will depend on the initial data and $M$.}

	\medskip
	
	\noindent \textbf{Claim.} 
 {For each 	$1 < q < \frac{11}{8}$,
	  if 
$\varepsilon>0$ is sufficiently small, then	 
	 \begin{equation*}\label{eq:J1}
		\begin{split}
			J_1 \lesssim P^{\frac12+ \varepsilon}(1 + t^{\frac12 (q-1+2 \varepsilon)\bt }). 
		\end{split}
	\end{equation*} Similarly, we have } \begin{equation*}\label{eq:J2}
	\begin{split}
		J_2 \lesssim P^{\frac32 -q+ \varepsilon }(1 + t^{\frac12 (q-1+2 \varepsilon)\bt}), \qquad 1 < q < \frac{17}{14}. 
	\end{split}
\end{equation*} 

\medskip 

\noindent 
We first check the estimate for $J_1$. This time, we bound using \eqref{eq:F-est} \begin{equation*}
	\begin{split}
		J_1 &= \iiiint_{(\Pi_+)^2} \int_0^{\frac{\pi}{2}} \frac{ (r \bR)^2 \bz \bX^{q-1} (\ct)^2}{(X\bX^2)^{\frac{q}{2}} Y^{\frac{3}{2}-q}    } \, \ud\tht \, \omg \bar{\omg} \\ 
		&\lesssim  \iiiint_{(\Pi_+)^2} (r \bR)^2 \bz^{2(q-1)} \, \frac{1}{d^{q-1}} \int_0^{\frac{\pi}{2}}   \frac{ (\ct)^2}{   X^{\frac12} \bX } \, \ud\tht \, \omg \bar{\omg} \\
		&\lesssim  \iiiint_{(\Pi_+)^2} (r \bR)^2 \bz^{2(q-1)} \, \frac{1}{d^{q-1}} \, \frac{(r\bR)^{\tau-\frac32}}{d^{2\tau}}  \, \omg \bar{\omg} .
	\end{split}
\end{equation*} Next, we use H\"older's inequality to bound \begin{equation*}
\begin{split}
	&\iiiint_{(\Pi_+)^2} (r \bR)^2 \bz^{2(q-1)} \, \frac{1}{d^{q-1}} \, \frac{(r\bR)^{\tau-\frac32}}{d^{2\tau}}  \, \omg \bar{\omg}  \\
	&\qquad \lesssim \left(  \iiiint_{(\Pi_+)^2} (r \bR)^2   \, \omg \bar{\omg}     \right)^{\frac14(1+2\tau)} \left(   \iint_{\Pi_+} \bz^{\frac{8(q-1)}{3-2\tau}} \bar{\omg} \, \iint_{\Pi_+} \frac{\omg}{d^{(q-1+2\tau)\frac{4}{3-2\tau}}} \right)^{\frac34 - \frac{\tau}{2}} 
\end{split}
\end{equation*} We now use Lemma \ref{lem:2D}: for $\tau>0$ small and $q<\frac{11}{8}$, \begin{equation*}
\begin{split}
	\sup_{ (\bR,\bz) \in \Pi_+ } \iint_{\Pi_+} \frac{|\omg|}{d^{\frac{4(q-1+2\tau)}{3-2\tau} }} \lesssim  \nrm{\omg}_{L^1(\Pi_+)}^{1-\frac{2(q-1+2\tau)}{(3-2\tau)} } (Mt^\beta)^{\frac{2(q-1+2\tau)}{(3-2\tau)}} \lesssim t^{\frac{2(q-1+2\tau)}{(3-2\tau)}\bt }, \qquad t\ge 1. 
\end{split}
\end{equation*} Then, as long as $\frac{8}{3}(q-1)<1$, we can take $\tau>0$ sufficiently small and estimate \begin{equation*}
\begin{split}
	 \iint_{\Pi_+} \bz^{\frac{8(q-1)}{3-2\tau} } |\bar{\omg}| \lesssim Z_0^{\frac{8(q-1)}{3-2\tau}} \nrm{\omg}_{L^1(\Pi_+)}^{1 -\frac{8(q-1)}{3-2\tau}} \lesssim 1. 
\end{split}
\end{equation*} This gives for $t\ge1$ \begin{equation*}
\begin{split}
	J_1 \lesssim P^{\frac{1}{2}(1+2\tau)} t^{ \frac{q-1+2\tau}{2}\bt}. 
\end{split}
\end{equation*} We now proceed to the estimate of $J_2$: using \eqref{eq:F-est} and H\"older's inequality, \begin{equation*}
\begin{split}
	 J_2 & = \iiiint_{(\Pi_+)^2} \int_0^{\frac{\pi}{2}} \frac{ (r \bR)^2 \bz Y^{q-1} (\ct)^2}{(X\bX^2)^{\frac{q}{2}} Y^{\frac{3}{2}-q}    } \, \ud\tht \, \omg \bar{\omg} \\
	 & \lesssim  \iiiint_{(\Pi_+)^2}  \frac{ (r \bR)^{3-q} \bz^{4(q-1)} }{d^{q-1}} \int_0^{\frac{\pi}{2}} \frac{ (\ct)^2}{X^{\frac12}\bX} \, \ud\tht \, \omg \bar{\omg} \\
	 &  \lesssim  \iiiint_{(\Pi_+)^2}  \frac{ (r \bR)^{\frac32-q + \tau} \bz^{4(q-1)} }{d^{q-1+2\tau}}  \, \omg \bar{\omg} \\
	 &  \lesssim \left(  \iiiint_{(\Pi_+)^2} (r \bR)^2   \, \omg \bar{\omg}     \right)^{ \frac{1}{2}(\frac32-q + \tau) } \left(   \iint_{\Pi_+} \bz^{\frac{16(q-1)}{ 1+2q-2\tau } }\bar{\omg} \, \iint_{\Pi_+} \frac{\omg}{d^{\frac{4(q-1+2\tau)}{1+2q-2\tau} } } \right)^{\frac12 (\frac{1}{2} + q - \tau)}.  
\end{split}
\end{equation*} As long as $q< \frac{17}{14}$, we can take $\tau>0$ sufficiently small that \begin{equation*}
\begin{split}
	 \iint_{\Pi_+} \bz^{\frac{16(q-1)}{ 1+2q-2\tau } }|\bar{\omg} | \lesssim 1, \qquad \sup_{ (\bR,\bz) \in \Pi_+ } \iint_{\Pi_+} \frac{|\omg|}{d^{\frac{4(q-1+2\tau)}{1+2q-2\tau} } }  \lesssim t^{ \frac{2(q-1+2\tau)}{1+2q-2\tau}\bt }. 
\end{split}
\end{equation*} This gives \begin{equation*}
\begin{split}
	J_2 \lesssim  P^{\frac{3}{2} - q + \tau } t^{ \frac{q-1+2\tau}{2}\bt}. 
\end{split}
\end{equation*} This finishes the proof of \textbf{Claim}.

\medskip 

 {
\noindent In the following, $\eps>0$ will denote a sufficiently small constant whose value could change from a line to another. Given \textbf{Claim}, we now observe, 
for any $1<q<17/14$,   \begin{equation*}
	\begin{split}
		J \lesssim J_1 + J_2 \lesssim P^{\frac12+ {\varepsilon}}(1+t)^{\frac 1 2 (q-1+2\varepsilon)\beta} 
	\end{split}
\end{equation*}
since $P$ is increasing
(recall \eqref{jj1j2} and the above for $J$).
Recall from Proposition \ref{prop:key-1} that \begin{equation*}
\begin{split}
	P \gtrsim (1+t)^{\frac{2}{15}-\eps}. 
\end{split}
\end{equation*} Hence, combining the previous two inequalities, we deduce \begin{equation*}
\begin{split}
	J \lesssim P^{\frac12 + \frac{15}{4}(q-1) \bt + \eps}. 
\end{split}
\end{equation*} Returning to the energy estimate \eqref{en_est}: \begin{equation*}
\begin{split}
	E \lesssim \dot{P}^{1-\frac{1}{q}} J^{\frac{1}{q}} 
\end{split}
\end{equation*} and applying the above upper bound for $J$, we obtain that \begin{equation*}
\begin{split}
	\dot{P} P^{\frac{(1/2)}{q-1} + \frac{15}{4} \bt+ \eps } \gtrsim 1, 
\end{split}
\end{equation*} which gives upon integration
after taking $q$ close to $17/14$ that \begin{equation}\label{contra}
\begin{split}
	P \gtrsim t^{ \frac{1}{\frac{10}{3} + \frac{15}{4} \bt  } - \eps}.
\end{split}
\end{equation}

 On the other hand, from the assumption on the vorticity maximum, the estimate \eqref{est_R_t_2} with Lemma \ref{lem:P-ineq-Linfty2} implies that  
 	\begin{equation}
		\begin{split}
			P(t) \lesssim R(t)^\delta \nrm{\omg(t,\cdot)}_{L^\infty (\bbR^3_+)}^{2-\delta}\lesssim (1+t)^{2\delta} t^{(2-\delta)\beta}
			\lesssim (1+t)^{2\delta+(2-\delta)\beta},
		\end{split}
	\end{equation} where $R(t)=  \sup\{ r : (r,z) \in \supp(\omg(t,\cdot)) \}$.  Therefore, we derive a contradiction to \eqref{contra} once \begin{equation}\label{ineq}
\begin{split}
	2\delta+(2-\delta)\beta <  \frac{1}{\frac{10}{3} + \frac{15}{4} \bt}. 
\end{split}
\end{equation} Note that, for each $\delta\in[0,3/20)$, we denote the unique positive root of $$(2\delta+(2-\delta)\beta)\cdot ( \frac{10}{3} + \frac{15}{4} \bt ) = 1$$  by $\bt_0=\bt_0(\delta)$. Observing that the previous inequality \eqref{ineq} holds for any $\bt\in(0,\bt_0)$, we are done.}
\end{proof}

\begin{remark} The explicit value of $\beta_0(\delta)$ is
	$$\beta_0(\delta)=
	\frac{8+5\delta}{9(\delta-2)}+\frac{\sqrt{536-1148\delta+845\delta^2}}{9(2-\delta)\sqrt{5}},$$ which is strictly positive whenever 
	$0\leq \delta<3/20$. In particular, it satisfies
	$$
	\beta_0(\delta)>  \beta_1(\delta) \quad\mbox{when}\quad 
	0\leq \delta<1/15,
	$$ with
	$\beta_1={\frac{\frac{2}{15}-2\delta}{2-\delta} }$ 
	from Corollary 
	\ref{cor:main2}.
\end{remark}

\section{Infinite gradient growth}\label{sec_Holder}

In this section, we   distinguish between 
the axial vorticity vector $\omega(\bold{x})$ and its scalar quantity $\omega^\theta(r,z)=r\xi(r,z)$, and we prove Theorem \ref{thm:gradient-growth}. 

\medskip

\noindent \textbf{Step 1}. \uline{Choice of initial data and contradiction hypothesis}: We consider $$\xi_{0}(r,z)\in C^\infty(\overline{\bbR_+^2}),\quad \overline{\bbR^2_+}:=\{r\geq 0, z\in\mathbb{R}\},$$ which is 
an odd function in $z$,
 satisfying the following properties on the first quadrant $\overline{(\bbR_+)^2}:=\{r\geq 0, z\geq 0\}$: \begin{itemize}
	\item $0\leq (-\xi_{0}) \leq 1$ on  $(\bbR_+)^2$,
	\item $\supp(\xi_{0})\cap  \overline{(\bbR_+)^2}\subset [0,1]^2$, 
	\item $\xi_{0} = -1$ on a set $\Omg_{0}\subset [0,1]^2$ satisfying $\int_{\Omega_0}  \,  r\ud r \ud z \ge \frac{9}{10} \int_{[0,1]^2} \,  r\ud r \ud z $,  
	\item $\xi_{0} = -z$ on $[0,1/2]\times[0,1/10]$. 
\end{itemize}
 It is not difficult to see that there exists such $\xi_{0}(r,z)$ such that the corresponding axial vorticity
$$\omega_0(\bold{x})=\omega^{\theta}_0e_\theta(\theta)=r\xi_0(r,z) e_\theta(\theta)$$  belongs to the class $C^\infty_c(\bbR^3)$. For instance, the troublesome  near the axis $\{r=0\}$ can be removable once we assume $$\xi_0(r,z)=a(z),\quad 0\leq r \leq r_0,\quad z\in\mathbb{R},$$ 
for some $r_0>0$ and for some function 
$a\in C^\infty(\mathbb{R})$ since
$$re_\tht(\tht)=(-r\sin \tht, r\cos\tht,0)=(-x_2, x_1,0)$$ is smooth in $\mathbb{R}^3$. Towards a contradiction, assume that the corresponding solution $\omg$ satisfies \begin{equation*}
\begin{split}
	\sup_{t\geq 0} \nrm{\omg(t,\cdot)}_{C^\alp(\bbR^3)} \le M
\end{split}
\end{equation*} for some $M>0$ and $\alp>0$. Our convention of the H\"older norm is simply \begin{equation*}
\begin{split}
	\nrm{\omg(t,\cdot)}_{C^\alp(\bbR^3)} &:= \nrm{\omg(t,\cdot)}_{L^\infty(\bbR^3)}  + \nrm{\omg(t,\cdot)}_{C^\alp_*(\bbR^3)}  \\
	&:= \sup_{\bold{x}\in\mathbb{R}^3} |\omg(t,\bold{x})| + \sup_{\bold{x} \ne \bold{x}'} \frac{|\omg(t,\bold{x})-\omg(t,\bold{x}')|}{|\bold{x}-\bold{x}'|^\alp}. 
\end{split}
\end{equation*}

\medskip

\noindent \textbf{Step 2}. \uline{Bounds on the solution}: To begin with, we observe that \begin{equation}\label{bounded_traj}
	\begin{split}
		 \sup_{(r,z) \in \Omg_{0}} \Phi^r(t,(r,z)) \le M, \quad t\geq 0,
	\end{split}
\end{equation} 
where $\Phi$ is the trajectory map \eqref{ode}.
This simply follows from \begin{equation*}
\begin{split}
	M \ge -\omg(t,\Phi(t,(r,z))) = -\Phi^r(t,(r,z)) \xi_{0}(r,z) = \Phi^r(t,(r,z))
\end{split}
\end{equation*} for $(r,z) \in \Omg_{0}$, from the definition of $\Omg_{0}$.  Next,
  we recall the Calderon--Zygmund estimate (cf. \cite{MB}) \begin{equation}\label{eq:velgrad}
\begin{split}
	\nrm{\nb u}_{C^\alp_*(\bbR^3)} \le C_\alp \nrm{\omg}_{C^\alp_*(\bbR^3)} \le C_\alp M. 
\end{split}
\end{equation}
\medskip
 {
\noindent \textbf{Step 3}. \uline{Non-trivial mass on a bounded region for all time}: 
 We claim that there exists an absolute constant $c>0$ such that\begin{equation}\label{claim:center}
\begin{split}
	 \int_{  \{ 0 \le z \leq 2,\,\, 0\leq r\leq M \}} \mathbf{1}_{\Omg(t)} \, r \ud r \ud z=\int_{  \{ 0 \le z \leq 2,\,\, r>0 \}} \mathbf{1}_{\Omg(t)} \, r \ud r \ud z \ge c>0,\quad t\geq 0. 
\end{split}
\end{equation} where $\Omg(t):=\Phi(t,\Omg_0)$.
Indeed, the   equality above simply follows from \eqref{bounded_traj}. For the inequality, we first observe
\begin{equation*}
\begin{split}
	 {\int_{(\bbR_+)^2} z(-\xi(t,r,z))  \, r \ud r \ud z}  &\leq  {\int_{(\bbR_+)^2} {z}(-\xi_{0}(r,z))  \, r \ud r \ud z}  \\ &\leq {\int_{(\bbR_+)^2} (-\xi_0(r,z))  \, r \ud r \ud z}={\int_{(\bbR_+)^2}  (-\xi(t,r,z))  \, r \ud r \ud z} ,
\end{split}
\end{equation*}
 where we used   monotonicity of 
the vorticity $z$--impulse (Lemma \ref{lem:z-down}) and the fact that  $ (-\xi_0)$ as a function on $\overline{(\bbR_+)^2}$ is non-negative and supported in $[0,1]^2$.
Then we estimate  \begin{equation*} 
\begin{split}
	  {\int_{(\bbR_+)^2}  (-\xi(t,r,z))  \, r \ud r \ud z}  &\geq 
	   {\int_{(\bbR_+)^2}  z(-\xi(t,r,z))  \, r \ud r \ud z}    \geq
	    {\int_{\{  z > 2,\,\,r>0\}}  z(-\xi(t,r,z))  \, r \ud r \ud z} \\ & \geq 
	2    {\int_{\{  z > 2,\,\,r>0\}}   (-\xi(t,r,z))  \, r \ud r \ud z} 
	\\ & = 2   {\int_{(\bbR_+)^2}  (-\xi(t,r,z))  \, r \ud r \ud z} -
	2    {\int_{\{  0\leq z \leq 2,\,\,r>0\}}   (-\xi(t,r,z))  \, r \ud r \ud z},   
\end{split}
\end{equation*} which implies  \begin{equation}\label{center_mass_est}
\begin{split}
	 {\int_{  \{ 0 \le z \leq 2,\,\, r>0 \}}  (-\xi(t,r,z))  \, r \ud r \ud z}  \geq \frac 1 2 {\int_{(\bbR_+)^2}  (-\xi(t,r,z))  \, r \ud r \ud z}.
\end{split}
\end{equation}
The right-hand term of \eqref{center_mass_est} has a lower bound:
\begin{equation*}
\begin{split}
  \frac 1 2 {\int_{(\bbR_+)^2}  (-\xi(t,r,z))  \, r \ud r \ud z} 
  &=  \frac 1 2 {\int_{(\bbR_+)^2}  (-\xi_0(r,z))  \, r \ud r \ud z} \\
  &\geq  \frac 1 2 {\int_{\Omega_0}   \, r \ud r \ud z}  
  \geq   \frac 1 2\cdot \frac 9 {10}{\int_{  [0,1]^2 }   \, r \ud r \ud z}
\end{split}
\end{equation*}  while the left-hand term of \eqref{center_mass_est} has an upper bound:
\begin{equation*}
\begin{split}
	{\int_{  \{ 0 \le z \leq 2,\,\, r>0 \}}  (-\xi(t,r,z))  \, r \ud r \ud z}  &= {\int_{  \{ 0 \le z \leq 2,\,\, r>0 \}\cap \Omega(t)}    \, r \ud r \ud z} + {\int_{  \{ 0 \le z \leq 2,\,\, r>0 \}\setminus\Omega(t)}  (-\xi(t,r,z))  \, r \ud r \ud z}. 
\end{split}
\end{equation*} For the last term above, we estimate
\begin{equation*}
\begin{split}
  {\int_{  \{ 0 \le z \leq 2,\,\, r>0 \}\setminus\Omega(t)}  (-\xi(t,r,z))  \, r \ud r \ud z}&\leq  {\int_{  \Omega(t)^c}  (-\xi(t,r,z))  \, r \ud r \ud z}  = {\int_{  \Omega_0^c}  (-\xi_0(r,z))  \, r \ud r \ud z}\\
 = & {\int_{ [0,1]^2\setminus \Omega_0}  (-\xi_0(r,z))  \, r \ud r \ud z}\leq  {\int_{ [0,1]^2\setminus \Omega_0}   \, r \ud r \ud z}\leq  \frac 1 {10} {\int_{ [0,1]^2 }   \, r \ud r \ud z}.
	 \end{split}
\end{equation*} 
Combining the above estimates, \eqref{center_mass_est} shows
 \begin{equation*}
\begin{split}
	\left(\frac 9 {20}-\frac 1 {10}\right)\cdot  {\int_{  [0,1]^2 }   \, r \ud r \ud z}     \leq  {\int_{  \{ 0 \le z \leq 2,\,\, r>0 \}\cap \Omega(t)}    \, r \ud r \ud z},
\end{split}
\end{equation*} which gives the claim \eqref{claim:center}
}
\medskip

\noindent \textbf{Step 4}. \uline{Velocity gradient near the origin}: 
 {
First, we recall the Biot-Savart law in $\mathbb{R}^3$ (\textit{e.g.} see \cite{MB}): $$u(t,\bold{x})=[\nabla\times (-\Delta)^{-1}\omega(t,\cdot)](\bold{x})
=\frac{1}{4\pi}\int_{\mathbb{R}^3} \frac{(\bold{x}-\bold{y})\times \omega(t,\bold{y})}{|\bold{x}-\bold{y}|^3}\ud \bold{y}.$$ We observe
$$
u^{x_3}(t,\bold{x})|_{x_1=x_2=0}=C \int_{\mathbb{R}^2_+}
\frac{r^2 \xi(t,r,z)}{(r^2+(x_3-z)^2)^{3/2}}r\ud r \ud z
$$  and 
$$
[\partial_{x_3}u^{x_3}](t,\bold{x})|_{x_1=x_2=0}=-C \int_{\mathbb{R}^2_+}
\frac{r^2 (x_3-z) \xi(t,r,z)}{(r^2+(x_3-z)^2)^{5/2}}r\ud r \ud z.
$$ Since $\xi$ is odd in $z$, we have, at the origin,
\begin{equation*}
	\begin{split}
		-\rd_{z} u^{z}(t,\mathbf{0})
		 = C \int_{(\bbR_+)^2} \frac{r^{2}z}{(r^2+z^2)^{\frac{5}{2}}}(-\xi(t,r,z) )\, r \ud r \ud z
	\end{split}
\end{equation*} 
for some absolute constant $C>0$. Now we estimate $-\rd_{z} u^{z}(t,\mathbf{0})$ from above and below.}
To begin with, we observe the uniform bound \begin{equation*}
\begin{split}
	0 \le -\rd_{z} u^{z}(t,\mathbf{0}) \le C(\nrm{\xi}_{L^1(\bbR^3)} + \nrm{\xi}_{L^\infty(\bbR^3)} )= C(\nrm{\xi_{0}}_{L^1(\bbR^3)} + \nrm{\xi_{0}}_{L^\infty(\bbR^3)} ) \le C. 
\end{split}
\end{equation*} This follows from considering separately the regions $\{ r^2+z^2 \le 1 \}$ and $\{ r^2+z^2 > 1 \}$.
We now proceed to obtain a lower bound. Since the integrand is non-negative, 
 we may estimate  {\begin{equation}\label{eq:velgrad-lb}
\begin{split}
	-\rd_{z} u^{z}(t,\mathbf{0}) & \geq  C\int_{(\bbR_+)^2}\frac{r^{2}z}{(r^2+z^2)^{\frac{5}{2}}}  \mathbf{1}_{\Omg(t)} \, r \ud r \ud z\\
	& \geq  \frac{C}{(M^2 + 2^2)^{\frac{5}{2}}}\int_{  \{ 0 \le z \leq 2,\,\, 0\leq r\leq M\} }  r^2 z  \mathbf{1}_{\Omg(t)} \, r \ud r \ud z \ge C_M,
\end{split}
\end{equation} where the last inequality follows from   \eqref{claim:center}, and  $C_M>0$ is a constant depending only on $M$.}
 Next, recall that from  {the divergence-free condition $	\rd_r (r u^r) + \rd_z (r u^z) = 0$,} we have \begin{equation*}
	\begin{split}
		2\rd_{r} u^{r}(t,\mathbf{0})= -\rd_{z} u^{z}(t,\mathbf{0})>0.
	\end{split}
\end{equation*} 
Therefore, using \eqref{eq:velgrad-lb} with \eqref{eq:velgrad}, we can take  {sufficiently small} $L = L(M,\alp)>0$ such that for $(r,z) \in [0,L]^2$, \begin{equation*}
\begin{split}
	-\frac{5}{4}\rd_{z} u^{z}(t,\mathbf{0}) > -\rd_{z} u^{z}(t,r,z) > -\frac{3}{4}\rd_{z} u^{z}(t,\mathbf{0})
\end{split}
\end{equation*} and \begin{equation*}
\begin{split}
	\frac{5}{4}\rd_{r} u^{r}(t,\mathbf{0}) > \rd_{r} u^{r}(t,r,z) > \frac{3}{4}\rd_{r} u^{r}(t,\mathbf{0})
\end{split}
\end{equation*} uniformly for all $t\ge0$. In particular, on $[0,L]^2$, it is guaranteed that the velocity field is pointing southeast for all times  {since $u^r|_{r=0}\equiv 0 $ and $u^z|_{z=0}\equiv 0$}. 

\medskip

\noindent \textbf{Step 5}. \uline{Growth of  {the H\"older norm}}: From the equation for $\xi$, we obtain that \begin{equation*}
	\begin{split}
	 {	\rd_t \frac{\xi}{z^\alpha} + u\cdot\nb \frac{\xi}{z^\alpha}= - \frac{u^{z}}{\alpha z} \frac{\xi}{z^\alpha}. }
	\end{split}
\end{equation*} We take the point $x^\eps_0 := (\eps, L)$ with $0<\eps<L$. As long as  {$x^\eps(\tau):=\Phi(\tau,x^\eps_0) \in [0,L]^2$ for all $\tau\in[0,t]$,} we have from integrating the above that \begin{equation*}
\begin{split}
	 {-\frac{\xi}{ {z^\alpha}} (t,x^\eps(t)) = \exp\left( - \int_0^t  \frac{u^z}{ {\alpha z}} (\tau,x^\eps(\tau))\,\ud \tau  \right) \cdot\left(-\frac{\xi_0}{ {z^\alpha}} (x^\eps_0)\right)\ge L^{1-\alpha}\exp\left( c t   \right),\quad {c=c(M,\alpha)>0.}}
\end{split}
\end{equation*} Here, we could have assumed that $L>0$ is sufficiently small, so that $\xi_0=-z$ on $[0,L]^2$.  {Let $0<T_\eps<\infty$} be the first time when $\Phi^r(t,x^\eps_0) = L$. 
 {The definition of $T_\eps$ makes sense since the trajectory is moving to the southeast and cannot touch the axis $\{z=0\}$ due to the unique solvability of the O.D.E. \eqref{ode}.}
From the uniform bound on $\rd_r u^r$ on $[0,L]^2$ and the fact that $u^r(t,0,z)=0$, we have that $T_\eps \to\infty$ as $\eps\to0$. This shows that \begin{equation*}
\begin{split}
 {	-\frac{\xi}{z^\alpha} (T_\eps, x^\eps(T_\eps ))  \ge L^{1-\alpha}\exp\left( cT_\eps  \right)}
\end{split}
\end{equation*} can be as large as we want. But then from \begin{equation*}
\begin{split}
	 {	\nrm{\omg}_{C^\alpha} \ge -\frac{\omg^\theta}{z^\alpha} (T_\eps, x^\eps(T_\eps ))  = -L\frac{\xi}{z^\alpha} (T_\eps, x^\eps(T_\eps )) , }
\end{split}
\end{equation*} we obtain a contradiction by taking $\eps\to0$.

\section{Enstrophy growth for Navier--Stokes solutions}\label{sec_navier}
 
In this section, we demonstrate that infinite enstrophy growth for Euler can be translated to \textbf{enstrophy inflation} for Navier--Stokes with small viscosity. While this is strictly speaking not necessary, we restrict ourselves to the axisymmetric systems without swirl: in this case, the 3D Navier--Stokes equations simplify to \begin{equation}\label{eq:NSE}
	\begin{split}
		\rd_t \left( \frac{\omg}{r} \right) + u \cdot \nb \left( \frac{\omg}{r} \right)  = \nu\left( \lap + \frac{2}{r}\rd_r \right)\left( \frac{\omg}{r} \right)  , \qquad u = K[\omg]. 
	\end{split}
\end{equation} In the statement below, we assume that $\omg_0, \omg_0^\nu$ are axisymmetric axial vorticities. 
\begin{proposition}
	Let $\omg_0 \in L^\infty (\bbR^3)$ be compactly supported with finite kinetic energy. Furthermore, assume that the corresponding Euler solution \eqref{eq:Euler-axisym-no-swirl} exists globally in time with $u \in L^\infty_{t,loc} W^{1,\infty}_x$ and \begin{equation*}
		\begin{split}
			\limsup_{t\to\infty} \nrm{\omg(t,\cdot)}_{L^2(\bbR^3)} = +\infty. 
		\end{split}
	\end{equation*} Then, for any sequence of initial data $\omg_0^\nu \in L^\infty \cap L^1 (\bbR^3)$ with finite kinetic energy which converges to $\omg_0$ strongly in $L^2(\bbR^3)$ as $\nu\rightarrow 0^+$, we have that \begin{equation*}
	\begin{split}
		\liminf_{\nu\to 0^+} \sup_{t \in [0,\infty]} \nrm{\omg^\nu(t,\cdot)}_{L^2(\bbR^3)} = +\infty, 
	\end{split}
\end{equation*} where $\omg^\nu$ is the solution to \eqref{eq:NSE} with initial data $\omg^\nu_0$ and viscosity $\nu$. 
\end{proposition}

Observe that 
  {Corollary \ref{cor:main2}
(\textit{e.g.} see \eqref{ex_l2_growh} of Remark \ref{rem_strange})} provides a class of Euler initial data for which the hypothesis of the above result is satisfied. Global in time existence and uniqueness of strong solutions with initial data $u_0 \in H^1(\bbR^3)$ to the axi-symmetric Navier--Stokes equations without swirl \eqref{eq:NSE} is well-known (\cite{UI}). 
\begin{proof}
	Towards a contradiction, assume that there exists a sequence $\nu_n\to 0^+$ such that \begin{equation*}
		\begin{split}
			\sup_{n\ge 1}\sup_{t \in [0,\infty]} \nrm{\omg^{\nu_n}(t,\cdot)}_{L^2(\bbR^3)} \le M
		\end{split}
	\end{equation*} for some $M>0$. We take some $T>0$, and since the sequence $\{ \omg^{\nu_n} \}_{n\ge 1}$ is precompact in the weak star topology of $L^\infty([0,T];L^2(\bbR^3))$, by passing to a subsequence if necessary, we derive \begin{equation*}
	\begin{split}
		\omg^{\nu_n}  \longrightarrow \hat{\omg}
	\end{split}
\end{equation*} weakly in star in $L^\infty([0,T];L^2(\bbR^3))$. It is not difficult to see that $\hat{\omg}$ is a weak solution to \eqref{eq:Euler-axisym-no-swirl}. Furthermore, using that $\omg^{\nu}(t,\cdot)$ is weakly continuous in time with values in $L^2(\bbR^3)$, it can be checked that $\hat{\omg}(t,\cdot) \rightarrow \omg_0$ as $t\to 0^+$ weakly in $L^2(\bbR^3)$. From weak-strong uniqueness of Euler in the class $u \in L^\infty_t (W^{1,\infty} \cap L^2)_{x}$ (see \cite{Danchin,CDL2}), we deduce that $\hat{\omg} = \omg$, and this is a contradiction since weak convergence implies \begin{equation*}
\begin{split}
	\sup_{t \in [0,T]} \nrm{\omg(t,\cdot)}_{L^2} \le M
\end{split}
\end{equation*} for all $T>0$. \end{proof} 

\section*{Acknowledgement}

\noindent KC has been supported by the National Research Foundation of Korea (NRF-2018R1D1A1B07043065)  and by the UBSI Research Fund(1.219114.01) of UNIST. IJ has been supported  by the New Faculty Startup Fund from Seoul National University and the Samsung Science and Technology Foundation under Project Number SSTF-BA2002-04. We thank Profs. Gianluca Crippa, Anna Mazzucato, and Vladimir \v{S}ver\'{a}k for helpful discussion and providing us several references. 

\bibliographystyle{plain}

\end{document}